\documentclass[11pt]{amsart}
\usepackage{amsmath}
\usepackage{amsfonts}
\usepackage{amssymb}
\usepackage{graphicx}
\usepackage{epstopdf}
\usepackage{aliascnt}
\usepackage{overpic}
\usepackage[bookmarks=true,pdfstartview=FitH, pdfborder={0 0 0}, colorlinks=true,citecolor=blue, linkcolor=blue]{hyperref}

\newtheorem{theorem}{Theorem}[section]
\newaliascnt{conj}{theorem}
\newaliascnt{cor}{theorem}
\newaliascnt{lemma}{theorem}
\newaliascnt{fact}{theorem}
\newaliascnt{claim}{theorem}
\newaliascnt{prop}{theorem}
\newaliascnt{definition}{theorem}

\newtheorem{cor}[cor]{Corollary}

\newtheorem{definition}[definition]{Definition}

\aliascntresetthe{conj} \aliascntresetthe{cor}
\aliascntresetthe{lemma} \aliascntresetthe{fact}
\aliascntresetthe{claim} \aliascntresetthe{prop}
\aliascntresetthe{definition}

\theoremstyle{remark}
\newaliascnt{remark}{theorem}
\newtheorem{remark}[remark]{Remark}
\aliascntresetthe{remark}

\theoremstyle{remark}
\newaliascnt{exam}{theorem}

\aliascntresetthe{exam}

\def\sek~{\S{}}

\makeatletter
\def\blfootnote{\gdef\@thefnmark{}\@footnotetext}
\makeatother

\newcommand{\std}{\operatorname{std}}
\newcommand{\ot}{\operatorname{ot}}
\newcommand{\tb}{\operatorname{tb}}
\newcommand{\const}{\operatorname{const}}
\newcommand{\trun}{\operatorname{trun}}
\newcommand{\In}{\operatorname{in}}
\newcommand{\Out}{\operatorname{out}}

\newcommand{\mangle}{\measuredangle}
\newcommand{\p}{\partial}

\newcommand{\CC}{\mathbb{C}}
\newcommand{\DD}{\mathbb{D}}
\newcommand{\RR}{\mathbb{R}}
\newcommand{\ZZ}{\mathbb{Z}}

\newcommand{\FF}{\mathcal{F}}
\newcommand{\SB}{\mathcal{SB}}

\newcommand{\bdn}{\mathbf{n}}
\newcommand{\bds}{\mathbf{s}}

\begin{document}

\title[on plastikstufe, bLob and overtwisted contact structures]{On plastikstufe, bordered Legendrian open book and overtwisted contact structures}
\author{Yang Huang}

\blfootnote{The author is supported by the Center of Excellence Grant ``Centre for Quantum Geometry of Moduli Spaces'' from the Danish National Research Foundation (DNRF95).}

\begin{abstract}
	In this paper we prove the presence of an embedded plastikstufe implies overtwistedness of the contact structure in any dimension. Moreover, we show in dimension 5 that the presence of an embedded bordered Legendrian open book (bLob) also implies overtwistedness. Applications of these criteria to certain (contact) fibered connected sums and open books are discussed.
\end{abstract}

\maketitle

\section{introduction} \label{sec:intro}
A contact structure $\xi$ is a hyperplane distribution in an odd-dimensional manifold $M$ which is maximally non-integrable. The questions of the existence and classification of contact structures lie in the center of contact topology. It turns out that among the set of all contact structures, there is a subset of ``flexible'' contact structures, known as {\em overtwisted} contact structures, which satisfy a form of $h$-principle in the sense of M. Gromov. In particular, the existence and classification of overtwisted contact structures can be dealt with using algebraic topology. 

In dimension 3, the overtwisted contact structures are introduced and classified by Y. Eliashberg in \cite{Eli89}. Recently this result is generalized to arbitrary dimensions by M. Borman, Y. Eliashberg and E. Murphy in \cite{BEM15}. Soon after the discovery of overtwistedness in higher dimensions, several useful criteria of overtwistedness are introduced by R. Casals, Murphy and F. Presas in \cite{CMP15}, where the overtwisted condition is connected with various commonly used notions in contact topology including open book decompositions, loose Legendrian submanifolds \cite{Mur12}, neighborhood size of overtwisted submanifolds \cite{NP2010}, contact surgery \cite{CE2012} and plastikstufe \cite{Nie06,MNPS13}. In particular we focus on the following (a bit technical) criterion for overtwistedness. One of the main goals of this paper is to remove the technical conditions of the following theorem from \cite{CMP15}, whose proof relies on the earlier work of \cite{MNPS13}.

\begin{theorem}[\cite{CMP15}] \label{thm:PS_criterion_old}
	A contact manifold is overtwisted if it contains a small plastikstufe with spherical core and trivial rotation.
\end{theorem}

Plastikstufe, originally introduced by K. Niederkr\"uger in \cite{Nie06} (see also \cite{Gro85}), is a certain submanifold of a contact manifold which can be thought of as a parametrized family of overtwisted disks in dimension 3. See \autoref{subsec:PS} for more careful definition of the plastikstufe. According to \cite{Nie06}, the presence of an embedded plastikstufe obstructs (strong) symplectic fillings. Let us remark that the requirement of the plastikstufe to be small, with spherical core and trivial rotation as stated in \autoref{thm:PS_criterion_old} is purely technical and will not be needed in this paper. So we refer the interested reader to \cite{MNPS13} for more details. It should be mentioned that by the work of Murphy, Niederkr\"uger, O. Plamenevskaya and A. Stipsicz \cite{MNPS13}, it is known that a Legendrian submanifold is loose, in the sense of Murphy \cite{Mur12}, if its complement contains an embedded small plastikstufe with spherical core and trivial rotation. Then \autoref{thm:PS_criterion_old} follows from the characterization of overtwistedness using loose Legendrian unknot. See \autoref{subsec:loose criteria} for more discussions on the looseness of Legendrian submanifolds.

Out first main result removes the technical conditions in \autoref{thm:PS_criterion_old}.

\begin{theorem} \label{thm:PS_criterion_new}
	A contact manifold is overtwisted if it contains an embedded plastikstufe.
\end{theorem}

Let us note that by the overtwisted $h$-principle established in \cite{BEM15}, one can always find embedded plastikstufes in overtwisted contact manifolds.

Later, the notion of plastikstufe is generalized by P. Massot, Niederkr\"uger and C. Wendl in \cite{MNW13} to the so-called {\em bordered Legendrian open books}, or bLobs in short, which is, roughly speaking, a compact submanifold of $(M,\xi)$ with non-empty Legendrian boundary and comes with an open book decomposition whose pages are also Legendrian. In particular, the plastikstufe may be regarded as a bLob such that the page is diffeomorphic to the product of the binding and a closed interval. See \autoref{subsec:PS} for more details about the definition of bLobs. 

For contact 5-manifolds, \autoref{thm:PS_criterion_new} can be strengthened as follows.

\begin{theorem} \label{thm:bLob_criterion}
	A contact 5-manifold is overtwisted if it contains an embedded bLob.
\end{theorem}

It is shown in \cite[Theorem 4.4]{MNW13} that the presence of an embedded bLob obstructs weak (semipositive) symplectic fillings given a technical cohomological condition on the symplectic form. Combining with \autoref{thm:bLob_criterion} and the overtwisted $h$-principle of course, one can easily remove the cohomological condition, at least in dimension 5.

It would be very interesting to see if the methods in this paper, outlined in \autoref{sec:outline}, can be applied to deal with bLobs in any dimension. In fact, the following corollary follows easily from \autoref{thm:PS_criterion_new} and \autoref{thm:bLob_criterion}.

\begin{cor} \label{cor:higher_dim}
	A contact manifold is overtwisted iff it contains an embedded bLob whose page is diffeomorphic to $W \times F$, where $W$ is a closed manifold and $F$ is a compact surface with (disconnected) boundary.
\end{cor}

Finally, to give some applications of the results proved in this paper, we study two examples, both of which are well-known to the experts. The first is the (contact) fibered connected sum operation suggested by Gromov and carried out by H. Geiges in \cite{Gei97}. In particular \autoref{thm:OT_fibered_sum} generalizes a result of Casals-Murphy \cite[Corollary 2]{CM16}. The second example concerns open books with monodromy equal to a negative power of the Dehn twist. In particular, we generalize a result of \cite{CMP15} which states that contact structures supported by a negatively stabilized open book is overtwisted. The idea that such objects should be related to overtwistedness goes back to E. Giroux. Our result is also consistent with a computation of Bourgeois-van Koert \cite{BvK10} on the vanishing of the contact homology. Although not studied in this paper, let us remark that there is another well-known method to produce plastikstufes, at least in dimension 5, using a generalized Lutz twist introduced by Etnyre-Pancholi \cite{EP11}.

Since the first appearance of this paper, there has appeared the work of J. Adachi \cite{Ada1609,Ada1610} which, in particular, proves an analogue of \autoref{thm:PS_criterion_old} using plastikstufes with certain torus core. He also give a construction of yet another Lutz-type twist which produces such plastikstufes.

The paper is organized as follows. We first outline the main ideas involved in the proof of the above results in \autoref{sec:outline}. In \autoref{sec:prelim} we review some background knowledge in (higher-dimensional) contact topology which is relevant for our purposes. This includes in particular the criteria of overtwistedness and the definition of plastikstufe and (slit) bLob. \autoref{sec:Legendrian_isotopy} is a technical section where we construct two special Legendrian isotopies, of type I and II, in a neighborhood of an overtwisted disk dimension 3. This will be crucial for our later constructions of the loose unknot. Then we prove \autoref{thm:PS_criterion_new} and \autoref{thm:bLob_criterion} in dimension 5 in \autoref{sec:PS_to_OT} and \autoref{sec:bLob_to_OT}, respectively. The proof of \autoref{thm:PS_criterion_new} for higher dimensional plastikstufe, as well as the proof of \autoref{cor:higher_dim}, are given in \autoref{sec:PS_to_OT_general}. Some applications are given in \autoref{sec:example}.\\

\noindent
\textsc{Acknowledgement:}
The author is deeply grateful to Ko Honda for numerous discussions and generously sharing his ideas on contact topology. He warmly thanks Emmy Murphy for making valuable comments on an earlier version of the paper, and Klaus Niederkr\"uger for explaining to the author previous works on overtwisted contatct structures. He thanks a referee for suggesting the applications in \autoref{sec:example}, and (possibly) another referee for detailed comments on both the mathematical and organizational part of the paper which greatly improves the level of rigor and clarity. Thanks also go to Jian Ge for useful communications. This work was initiated and partially done when the author visited Caltech in the spring 2016, and he thanks the department of Mathematics at Caltech for their hospitality.

\section{Outline of the main ideas} \label{sec:outline}

The idea of the proof of \autoref{thm:PS_criterion_new} and \autoref{thm:bLob_criterion} can be summarized in the following three steps, with increasing level of complexity. 

\textsc{Step 1}. 
Let $M$ be an overtwisted 3-manifold with a fixed contact form $\alpha_{\ot}$, e.g. a collar neighborhood of an overtwisted disk, and consider the associated higher-dimensional contact manifold $M \times T^\ast \RR^m, m \geq 1$. Let $(q_1,\cdots,q_m;p_1,\cdots,p_m)$ be the usual coordinates on $T^\ast \RR^m$. It is proved in \cite{CMP15} that $M \times T^\ast \RR^m$ equipped with the contact form 
	\begin{equation*}
		\alpha_{\ot} + \frac{1}{2} \sum_{i=1}^m ( q_i dp_i - p_idq_i )
	\end{equation*}
is overtwisted in the sense of \cite{BEM15}. From this it is deduced in \cite{CMP15} that a contact manifold (of dimension greater than 3) is overtwisted if and only if the standard Legendrian unknot is loose. Our starting point is to consider $M \times T^\ast \RR^m$ equipped with a different contact form 
	\begin{equation*}
		\alpha_{\ot} - \sum_{i=1}^m p_idq_i
	\end{equation*}
and an $\epsilon$-neighborhood $U_{\epsilon}$ of $M \times \RR^m \subset M \times T^\ast \RR^m$, where $\RR^m$ is identified with the 0-section in $T^\ast \RR^m$. We claim that for any $\epsilon>0$, there exists a loose Legendrian sphere in $U_{\epsilon}$ which is Legendrian isotopic to the standard unknot. To achieve this, we choose a (1-dimensional) stabilized Legendrian unknot $K_0 \in M \times \{0\}$ with $\tb(K_0)=-1$, whose existence is guaranteed by the overtwistedness of $M$. Then along every radial directions in $\RR^m$, we ``slowly'' Legendrian isotop $K_0$ to the standard Legendrian unknot in $M$. Here ``slow'' means, roughly speaking, the derivative of the isotopy, measured against $\alpha_{\ot}$, is small with respect to $\epsilon$. Finally we cap off the Legendrian by a family of standard (2-dimensional) Legendrian hemispheres (cf. \autoref{fig:cap}). Then the Legendrian lift of the totality of the Legendrian isotopies in all radial directions of $\RR^m$ together with the capping-off disks gives the desired loose Legendrian sphere which turns out to be Legendrian isotopic to the standard unknot.

\textsc{Step 2}. 
We first restrict ourselves to dimension 5, and consider the plastikstufe $P_S$ with core $S^1$. By Step 1, given any overtwisted contact 3-manifold $M$, an arbitrarily small neighborhood of $M \times \RR \subset M \times T^\ast \RR$ is overtwisted, i.e., contains a loose unknot. Now if we replace $\RR$ by $S^1$, the idea is to wrap the Legendrian isotopy in $M$ around $S^1$ while making sure that the resulting (loose) Legendrian $S^2$ is embedded and Legendrian isotopic to the standard unknot. Here we will use the type I Legendrian isotopy constructed in \autoref{sec:Legendrian_isotopy}. Ideas in this step essentially give the proof of \autoref{thm:PS_criterion_new} in the 5-dimensional case. Let us also point out that similar ideas are known to R. Casals. In particular, he proved \cite{Cas} that the $\epsilon$-neighborhood $$\RR^3 \times S^1 \times (-\epsilon, \epsilon) \subset (\RR^3 \times T^\ast S^1, \ker (\alpha_{\ot} -pdq))$$ of $\RR^3 \times S^1$ is overtwisted for any $\epsilon>0$.

\textsc{Step 3}. 
Now we switch to the case of a general plastikstufe in any dimension. In this case, we have a family of 2-dimensional overtwisted disks parametrized by a closed manifold $S$, i.e., the core of the plastikstufe. Pick a point $u \in S$ and consider the exponential map $\exp_u: T_u S \to S$ with respect to some (complete) Riemannian metric on $S$. Then the loose Legendrian unknot will be, roughly speaking, the Legendrian lift of the totality of Legendrian isotopies (of type II) of $K_0$ along all geodesic rays emanating from $u$, together with capping-off disks as in Step 1. Let us remark that the method of using type II Legendrian isotopy works in any dimensions, but we use instead the type I isotopy in dimension 5 only for the reason that it is more adaptable to the case of bLobs.

\textsc{Step 4}.
Consider a general bLob which may not split as a product. In order for the previous strategy to work, one first needs to find a 2-dimensional disk $D_{\ot}$ in the bLob cutting through the binding, on which the characteristic foliation looks exactly like the overtwisted disk in dimension 3. Such a disk is easily found in a plastikstufe, but is not obvious in general bLobs. In the case of plastikstufe, we have, in fact, a family of overtwisted disks parametrized by the binding. Unfortunately this symmetry breaks down for general bLobs, so some trick is needed to ``wrap around'' the binding. Let us make two remarks here. Firstly, the monodromy of the open book causes no difficulty and in fact, we will construct a modified version of the bLob, called {\em slit bLob}, in \autoref{subsec:PS} which does not involve the monodromy but still implies overtwistedness. Secondly, our construction does not immediately produce a plastikstufe out of a general bLob, although it formally follows from the $h$-principle for overtwisted contact structures.

\section{overtwistedness criteria and plastikstufe}  \label{sec:prelim}
In this section, we review some recent development in higher-dimensional flexibility of contact structures which is relevant for our purposes. 

\subsection{Loose Legendrian and criteria for overtwistedness} \label{subsec:loose criteria}
We recall the definition of loose chart and loose Legendrian submanifolds in higher dimensions following \cite{Mur12}. First consider the standard contact 3-space $(\RR^3, \xi_{\std}=\ker (dz-ydx))$. Let $\gamma \subset (\RR^3,\xi_{\std})$ be a stabilized Legendrian arc whose front projection is as shown in \autoref{fig:kink}. Specifically, the front projection of $\gamma$ has a unique transverse double point and a unique Reeb chord of length $a$, called the {\em action} of the stabilization. Let $B$ be an open ball in $\RR^3$ containing $\gamma$ of action $a$ as defined above. 

Now consider the Liouville space $T^\ast \RR^{n-1}$ with the usual dual coordinates $\{q,p\}$ and Liouville form $-pdq$. Let 
	\begin{equation*}
		V_C = \{|p|< C, |q|<C\} \subset T^\ast \RR^{n-1}
	\end{equation*}
be an open poly-disk in $T^\ast \RR^{n-1}$. Note that $B \times V_C$ is an open subset of $(\RR^{2n+1},\xi_{\std})$, which contains the Legendrian submanifold $\Lambda=\gamma \times \{|q|<C,p=0\}$. The pair $(B \times V_C, \Lambda)$ is called a {\em loose chart} if $a/C^2<2$. Finally, a Legendrian submanifold $L \subset (M,\xi)$ is {\em loose} if there exists a Darboux chart $U \subset M$ such that the $(U,U \cap L)$ is contactomorphic to a loose chart.

\begin{figure}[ht]
	\begin{overpic}[scale=.3]{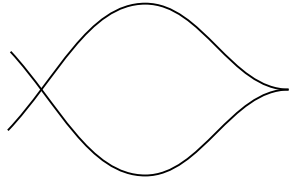}
	\end{overpic}
	\caption{The front projection of a stabilized Legendrian arc.}
	\label{fig:kink}
\end{figure}


Generalizing overtwisted contact structures in dimension 3, it is shown in \cite{BEM15} that overtwisted contact structures always exists in an almost contact manifold of any dimension and they obey a parametric $h$-principle. The precise content of the overtwisted $h$-principle is not particularly relevant for this paper so we refer the reader to \cite{BEM15} for more details.

Following \cite{CMP15}, it turns out that one can characterize overtwisted contact structures using loose Legendrian spheres discussed above. To be more precise, note that the standard contact sphere $(S^{2n-1},\xi_{\std})$ can be identified with the convex boundary of the unit ball $\DD^{2n} \subset \RR^{2n}$ equipped with the Liouville form $qdp-pdq$. Then the Lagrangian plane $\RR^n = \{p=0\}$ intersects $S^{2n-1}$ in a Legendrian sphere $K$, which we call the {\em standard Legendrian unknot}. By removing a point from $(S^{2n-1},\xi_{\std})$, we may view $K$ as a Legendrian sphere in $(\RR^{2n-1},\xi_{\std})$, which is also called the standard Legendrian unknot. Hence by Darboux's theorem, it makes sense to talk about standard Legendrian unknot in any contact manifold, well-defined up to Legendrian isotopy. 

Recall the following criterion for overtwistedness, which will be useful for our purposes.

\begin{theorem}[\cite{CMP15}] \label{thm:CMP}
	A contact manifold $(M,\xi)$ is overtwisted iff the standard Legendrian unknot is loose.
\end{theorem}

\subsection{Plastikstufe and bordered Legendrian open book} \label{subsec:PS}
We start with the notion of plastikstufe following \cite{Nie06}. Consider the overtwisted contact manifold $(\RR^3,\xi_{\ot})$ with contact form $\alpha_{\ot}$ given in cylindrical coordinates by 
	\begin{equation} \label{eqn:alpha_ot}
		\alpha_{\ot} = \cos r dz + r\sin r d\theta.
	\end{equation}
We can find an explicit overtwisted disk (cf. \autoref{fig:OTdisk})
	\begin{equation*}
		D_{\ot} = \{ z=0, r \leq \pi \} \subset (\RR^3,\xi_{\ot})
	\end{equation*}

Let $S$ be a closed manifold and consider $(T^\ast S,\lambda)$ equipped with the canonical Liouville form $\lambda$. In the following, we identify $S$ with the 0-section of $T^\ast S$. Then we can construct a new contact manifold $\RR^3 \times T^\ast S$ with contact form $\alpha = \alpha_{\ot} - \lambda$. By definition the submanifold $P_S = D_{\ot} \times S$ is a {\em plastikstufe} with core $S$.

Note that $P_S$ is equipped with a singular codimension-1 foliation $\FF_{P_S}$ which is defined by the kernel of $\alpha|_{P_S}$. The leaves of $\FF_{P_S}$ are Legendrian submanifolds by construction. According to \cite{Hua15}, the germ of the contact structure near $P_S$ is uniquely determined by $\FF_{P_S}$. Observe that the singular foliation $\FF_{P_S}$ is nothing but the open book decomposition with page $S \times [0,1]$ and identity monodromy such that $S \times \{0\}$ is identified with the binding. Inspired by this observation, we introduce the notion of bordered Legendrian open book, or bLob in short, as defined in \cite{MNW13}. Here we give the definitions in full generality but our results are applied only to bLobs in dimension 5.

\begin{definition} \label{defn:bLob}
	Given a compact $n$-manifold $\Sigma$ with disconnected boundary. Let $S \subset \p \Sigma$ be a non-empty proper subset of the connected components of $\p \Sigma$. A {\em bordered Legendrian open book (bLob)} $(\Sigma,\phi)$ in a $(2n+1)$-dimensional contact manifold is a compact $(n+1)$-dimensional submanifold
		\begin{equation*}
			Y = \Sigma \times [0,1] / (x,t) \sim (x,t')~ \forall x \in S, t,t' \in [0,1] \text{ and } (y,0) \sim (\phi(y),1) ~\forall y \in \Sigma
		\end{equation*}
	where $\phi: \Sigma \to \Sigma$ is a diffeomorphism which restricts to the identity map in a collar neighborhood of $S \subset \Sigma$, and such that all pages $\Sigma \times \{t\}$ and the boundary $\p Y$ are Legendrian submanifolds.
\end{definition}

It is shown in \cite{MNW13} (see also \cite{Hua15}) that the germ of contact structure near a bLob $(\Sigma,\phi)$ is uniquely determined by the Legendrian open book. Hence in order to prove \autoref{thm:bLob_criterion}, it suffices to show that the germ of contact structure near a bLob is overtwisted. Note that although one can easily write down an explicit contact form in a neighborhood of the plastikstufe, it is in general hard to write down a global contact form in a neighborhood of a bLob. However we will describe in \autoref{sec:bLob_to_OT} a contact form near a bLob by decomposing it into simpler pieces.

It turns out the information about the monodromy in a bLob plays no role in showing the overtwistedness in \autoref{sec:bLob_to_OT}. So we introduce below the notion of a {\em slit bLob}.

\begin{definition} \label{defn:slit_bLob}
	Given a compact manifold $\Sigma$ with disconnected boundary. Let $S \subset \p \Sigma$ be a non-empty proper subset of the connected components of $\p \Sigma$. Choose $\delta>0$ and identify $S \times [0,\delta]$ with a closed collar neighborhood of $S \subset \Sigma$ such that $S \times \{0\}$ is identified with $S$. A {\em slit bLob} is the following topological space
		\begin{equation*}
			W = \Sigma \times [0,1] / (x,t) \sim (x,t')~ \forall x \in S, t,t' \in [0,1] \text{ and } (y,0) \sim (y,1) ~\forall y \in S \times [0,\delta],
		\end{equation*}
	which is a smooth manifold away from $S \times \{\delta\} \times \{0\} \subset S \times [0,\delta] \times \{0\} \subset \Sigma \times \{0\}$. Moreover, we assume that $W$ admits an embedding into a contact manifold such that $\p W$ and the pages are all Legendrian.
\end{definition}

In other words, one can think of a slit bLob as a bLob partially cut open along a page where the monodromy is applied. It is therefore clear that the existence of a bLob implies the existence of a slit bLob but the other way around implication is not obvious. Although a slit bLob is not smooth, it has Legendrian boundary and the standard neighborhood theorem applies, i.e., the contact germ near a slit bLob is uniquely determined. Observe that the contact germ near a slit bLob is independent of the choice of the collar neighborhood of the binding and $\delta>0$, so we simply denote a slit bLob associated with $\Sigma$ by $\SB(\Sigma)$.

The reason for introducing the above definition is the following strengthening of \autoref{thm:bLob_criterion}.

\begin{theorem} \label{thm:slit_bLob_criterion}
	A contact 5-manifold is overtwisted iff it contains an embedded slit bLob.
\end{theorem}

It would be interesting to see if there are examples of contact manifolds in which a slit bLob can be found more easily than a bLob. The proof of \autoref{thm:slit_bLob_criterion} will be given in \autoref{sec:bLob_to_OT}.

\section{the standard overtwisted disk in dimension 3} \label{sec:Legendrian_isotopy}
This section if purely technical. The goal is to exhibit, in an arbitrarily small neighborhood of the standard overtwisted disk, two explicit Legendrian isotopies from a ``thick'' Legendrian unknot $J_{\text{thick}}$ with $\tb=-1$ to a ``thin'' Legendrian unknot $J_{\text{thin}}$. Roughly speaking, if we embedded the Legendrian unknot with $\tb=-1$ into a Darboux chart in the standard way, then the unique Reeb chord of $J_{\text{thick}}$ is much longer than the one of $J_{\text{thin}}$. See below for more precise definitions. 

As pointed out in \autoref{sec:outline}, these two Legendrian isotopies will be constructed with different properties which will suit our purposes of constructing the loose unknot in a neighborhood of plastikstufes of any dimension and bLobs in contact 5-manifolds, respectively.

\subsection{Construction of type I Legendrian isotopy} \label{subsec:type_I}
Recall from \autoref{subsec:PS} the overtwisted contact manifold $(\RR^3,\xi_{\ot})$ and the standard overtwisted disk 
	\begin{equation*}
		D_{\ot} = \{ z=0, r \leq \pi \} \subset (\RR^3,\xi_{\ot}) 
	\end{equation*}
Fix $\epsilon>0$, we consider an invariant collar neighborhood $D_{\ot} \times (-\epsilon,\epsilon)$ of $D_{\ot}$. Let 
	\begin{equation*}
		\Gamma = \{z=0,r=\pi/2\} \subset D_{\ot}
	\end{equation*}
be the circle along which the contact plane is perpendicular to $D_{\ot}$. Observe also that the rays $\{\theta=\const\}$ are all Legendrian. See \autoref{fig:OTdisk}.

\begin{figure}[ht]
	\begin{overpic}[scale=.3]{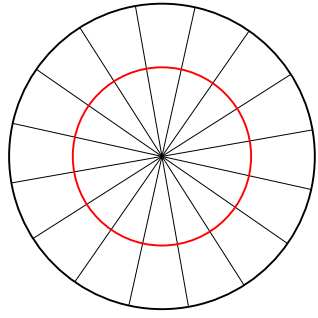}
	\end{overpic}
	\caption{The standard overtwisted disk. The red circle is $\Gamma$.}
	\label{fig:OTdisk}
\end{figure}

Given any $s \in (\pi/2,\pi]$, let $C_s \subset D_{\ot}$ be the circle of radius $s$. Parametrize $C_s$ by the angular coordinate $\theta_s \in (-\pi,\pi]$, where the subscript $s$ is to emphasize that the coordinate is considered on $C_s$. Clearly $C_\pi$ is Legendrian. For any $s \in (\pi/2,\pi)$, let $C'_s$ be the unique Legendrian passing through $z_{\theta=0}=0$ and whose projection to $D_{\ot}$ coincides with $C_s$. It is clear that $C'_s$ diffeomorphic to the real line. Here $z_\theta$ denotes the $z$-coordinate at angle $\theta$. Abusing notations, we also regard $\theta_s$ as a coordinate function on $C'_s$. Similarly the Legendrian $C'_s$ can be defined for $s \in (0,\pi/2)$.

Fix an arbitrarily small $\theta_0>0$. There exists $a_0$ smaller than but sufficiently close to $\pi$ such that the holonomy along the Legendrian arc $A_0 = \{ \theta_0 \leq \theta \leq 2\pi-\theta_0 \} \subset C'_{a_0}$, defined to be segment nontrivially intersecting the plane $\{z=0\}$, increases the $z$-coordinate by a positive $\delta \ll \epsilon$. Similarly there exists small $b_0>0$ such that the holomony along the Legendrian arc $B_0 = \{ -\theta_0 \leq \theta \leq \theta_0 \} \subset C'_{b_0}$ decreases the $z$-coordinate by $\delta$. Here the holonomies are measured in the counter-clockwise direction. Take Legendrian segments in the radial directions $R_{\pm \theta_0} = \{ \theta=\pm\theta_0, b_0 \leq r \leq a_0 \}$ on the planes $\{z=\mp \delta/2\}$, respectively. Patching the pieces together, we have now obtained the thick Legendrian loop $J_{\text{thick}} = A_0 \cup B_0 \cup R_{\theta_0} \cup R_{-\theta_0}$.

We define $J_{\text{thin}}$ in a similar way. Namely, there exists $a'_0 \in (\pi/2,\pi)$ such that the holonomy along the Legendrian arc $A'_0 = \{ \pi-\theta_0 \leq \theta \leq \pi+\theta_0 \}$ increases the $z$-coordinate by $\delta$. Also there exists $b'_0 \in (0,\pi/2)$ such that the holonomy along the Legendrian arc $B'_0 = \{ \theta_0-\pi \leq \theta \leq \theta_0+\pi \}$ increases the $z$-coordinate by $\delta$. Choosing the connecting Legendrian rays $R'_{\pm \theta_0}$ as before, we define the thin Legendrian loop $J_{\text{thin}} = A'_0 \cup B'_0 \cup R'_{\theta_0} \cup R'_{-\theta_0}$.

Finally to construct the Legendrian isotopy $J_t, t \in [0,1]$, from $J_{\text{thick}}$ to $J_{\text{thin}}$, it suffices to observe that for each choice of an angle between $\theta_0$ and $\pi-\theta_0$, the above construction defines a Legendrian loop. So we have obtained a (smooth if one ignores the corners) path of Legendrian loops connecting $J_{\text{thick}}$ and $J_{\text{thin}}$. See \autoref{fig:foliation1} for an illustration of this isotopy. 

The key point of the type I Legendrian isotopy is that the Legendrian loops in the isotopy from $J_{\text{thick}}$ to $J_{\text{thin}}$ are pairwise disjoint by construction. Moreover, notice that $\p_z$ is a contact vector field transverse to $D_{\ot}$ in $\RR^3_{\ot}$. Let $\psi_z^s$ be the time-$s$ flow of $\p_z$. Then it follows from the construction that $\psi_z^s (J_t) \cap \psi_z^{s'} (J_{t'}) = \emptyset$ if either $s \neq s'$ or $t \neq t'$.

\begin{figure}[ht]
	\begin{overpic}[scale=.35]{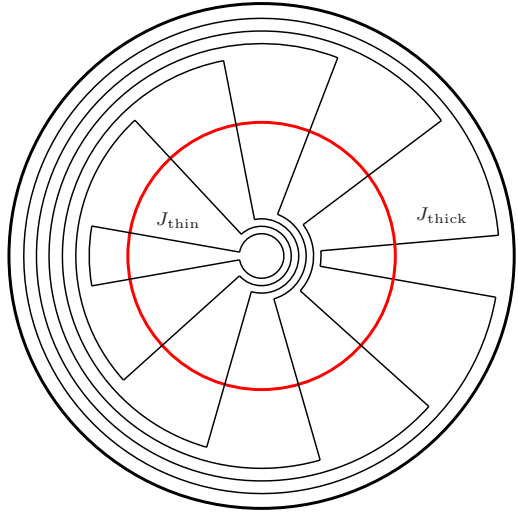}
		\put(80,57){\tiny{$J_{\text{thick}}$}}
		\put(29,56){\tiny{$J_{\text{thin}}$}}
	\end{overpic}
	\caption{Type I Legendrian isotopy from $J_{\text{thick}}$ to $J_{\text{thin}}$.}
	\label{fig:foliation1}
\end{figure}

\begin{remark} \label{rmk:corner_rounding}
	Note that the Legendrian knots $J_{\text{thick}}$ and $J_{\text{thin}}$ constructed above are only piecewise smooth. But there is a canonical way to approximate any piecewise smooth Legendrian knot by a smooth one by rounding the corners. Here canonical means that the Legendrian isotopy class of the resulting smooth Legendrian is independent of the choice of the approximation. Moreover the Thurston-Bennequin number of a piecewise smooth Legendrian knot is well-defined since the transverse Legendrian push-off exists. The same remark applies as well to the following section.
\end{remark}

\subsection{Construction of type II Legendrian isotopy} \label{subsec:type_II}
Now we turn to the construction of the type II Legendrian isotopy. The construction of type II isotopy is technically more involved than the type I isotopy in two aspects: firstly, the Legendrian loops in the isotopy are no longer pairwise disjoint which is the price we pay for making the isotopy positive, and secondly, we need to further shrink $J_{\text{thin}}$ all the way to a point with controlled rate of convergence for later applications. We continue to use the notations from the previous section.

Let $C''_{s,\theta}$ be the Legendrian curve whose projection to $D_{\ot}$ coincides with $C_s$ and such that $z_{\theta}=0$. As before, fix small $\theta_0>0$. Then there exists $a_0$ smaller than but sufficiently close to $\pi$ such that the holonomy along the Legendrian arc $A_0 = \{ \theta_0 \leq \theta \leq 2\pi \} \subset C''_{a_0,\theta=0}$ increases the $z$-coordinate by a positive $\delta \ll \epsilon$. Similarly there exists small $b_0>0$ such that the holonomy along the Legendrian arc $B_0 = \{ 0 \leq \theta \leq \theta_0 \} \subset C''_{b_0,\theta=0}$ decreases the $z$-coordinate by $\delta$. To complete the construction of $J_{\text{thick}}$, let us take Legendrian segments in the radial directions $R_0 = \{ \theta=0, b_0 \leq r \leq a_0 \}$ and $R_{\theta_0} = \{ \theta=\theta_0, b_0 \leq r \leq a_0 \}$ on the planes $\{z=0\}$ and $\{z=-\delta\}$, respectively. Define $J_{\text{thick}} = A_0 \cup B_0 \cup R_0 \cup R_{\theta_0}$.

In order to construct $J_{\text{thin}}$ as well as the isotopy, pick $0 < \nu < \mu \ll \theta_0$ and reparametrize the intervals by $\mu_t \in [2\pi,2\pi+\mu]$ and $\nu_t \in [\theta_0,2\pi+\nu]$, where $t \in [0,1]$, such that $\mu_0=2\pi, \mu_1=2\pi+\mu, \nu_0=\theta_0, \nu_1=2\pi+\nu$ and $\nu_t < \mu_t, \dot{\nu}_t > \dot{\mu}_t$ for all $t$. Define $A_t = \{ \nu_t \leq \theta \leq \mu_t \} \subset C''_{a_0,\theta=\mu_t}$. Let $\delta_t$ be the difference in the $z$-coordinate at the endpoints of $A_t$. There exists $b_t>0$, strictly decreasing in $t$, such that if $B_t = \{ \mu_t-2\pi \leq \theta \leq \nu_t \} \subset C''_{b_t,\theta=\mu_t}$ denotes the Legendrian segment, then the holonomy along $B_t$ decreases the $z$-coordinate by $\delta_t$. Taking the truncated Legendrian rays $R_{\mu_t} = \{ \theta=\mu_t, b_t \leq r \leq a_0 \}$ and $R_{\nu_t} = \{ \theta=\nu_t, b_t \leq r \leq a_0 \}$ on the planes $\{ z=0 \}$ and $\{ z=-\delta_t \}$, we can define $J_{\text{thin}} = A_1 \cup B_1 \cup R_{\mu_1} \cup R_{\nu_1}$ and the Legendrian isotopy $J_t = A_t \cup B_t \cup R_{\mu_t} \cup R_{\nu_t}$ for $t \in [0,1]$. See \autoref{fig:foliation2} for an illustration of the type II isotopy.

\begin{figure}[ht]
	\begin{overpic}[scale=.35]{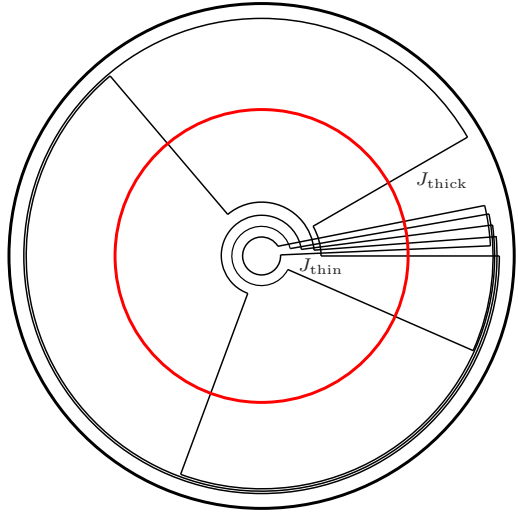}
		\put(80,64){\tiny{$J_{\text{thick}}$}}
		\put(57,47){\tiny{$J_{\text{thin}}$}}
	\end{overpic}
	\caption{Type II Legendrian isotopy from $J_{\text{thick}}$ to $J_{\text{thin}}$.}
	\label{fig:foliation2}
\end{figure}

Finally we want to further isotop $J_{\text{thin}}$ to a point with controlled rate of convergence as follows. Note that $J_{\text{thin}}$ bounds a disk $D_{\text{thin}} \subset D_{\ot} \times (-\epsilon,\epsilon)$ which is foliated by Legendrian arcs, i.e., the characteristic foliation on $D_{\text{thin}}$ has precisely two half-elliptic singularities on $J_{\text{thin}}$, and moreover, the projection of $D_{\text{thin}}$ to $D_{\ot}$ coincides with the obvious disk bounded by the projection of $J_{\text{thin}}$ to $D_{\ot}$. Now a standard neighborhood $D_{\text{thin}} \times (-\epsilon',\epsilon')$ of $D_{\text{thin}}$ determined by the characteristic foliation is a Darboux ball in which the Legendrian loops $J_t$, $t$ sufficiently close to $1$, can be drawn in the front projection as in \autoref{fig:up_isotopy}(a). Here the corners are rounded as in \autoref{subsec:type_I} (cf. Remark \ref{rmk:corner_rounding}). Then it is clear from the front projection that the Legendrian isotopy near $J_{\text{thin}}$ can be extended by a positive Legendrian isotopy $J_t, t \in [1,2)$, such that $J_1=J_{\text{thin}}$ and $\displaystyle{\lim_{t \to 2} J_t}$ is a single point. Moreover, we can arrange so that the trace
	\begin{equation*}
		\bigcup_{t=1}^2 ( J_t \times \{t\} ) \subset D_{\ot} \times (-\epsilon',\epsilon') \times \RR \subset D_{\ot} \times (-\epsilon',\epsilon') \times T^\ast \RR,
	\end{equation*}
defines the front projection of a smooth Legendrian disk in a 5-dimensional Darboux chart. See \autoref{fig:up_isotopy}(b).

\begin{figure}[ht]
	\begin{overpic}[scale=.4]{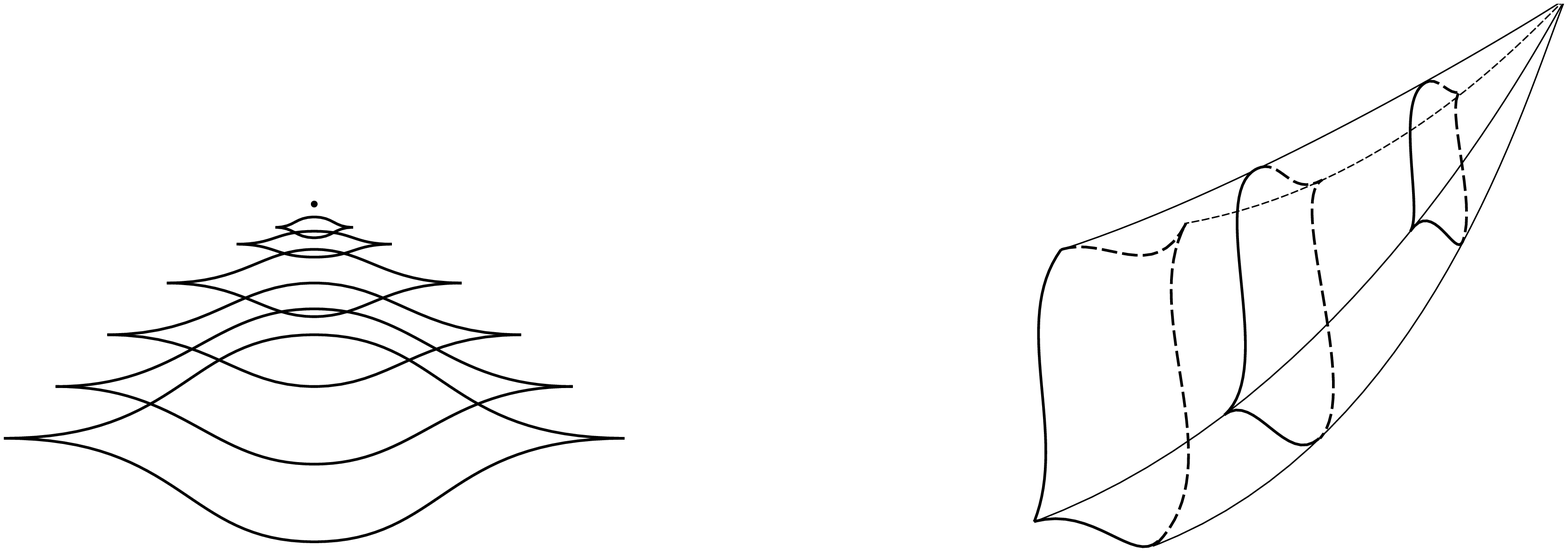}
		\put(40,6){$J_{t<1}$}
		\put(33.5,13){$J_1=J_{\text{thin}}$}
		\put(25,19){$J_{t>1}$}
		\put(18.3,-4){(a)}
		\put(82,-4){(b)}
	\end{overpic}
	\vspace{5mm}
	\caption{(a) The extended type II Legendrian isotopy from $J_{\text{thin}}$ to a point. (b) The front projection of the trace of the extended type II Legendrian isotopy.}
	\label{fig:up_isotopy}
\end{figure}

The key point for the (extended) type II isotopy is that it is a positive isotopy in the sense that $\alpha_{\ot} (\dot{J}_t)>0$ for all $t$, where $\alpha_{\ot}$ is as defined in (\ref{eqn:alpha_ot}). This is essentially because $\dot{J}_t$ can be decomposed into a component tangent to $\xi_{\ot}$ and a component in the positive direction of $\p_{\theta}$, which is positive with respect to $\alpha_{\ot}$. Moreover for any two distinct $t,t' \in [0,1]$, $J_t$ intersects (transversely) $J_{t'}$ in at most one point, and at $p = J_t \cap J_{t'}$ if exists, we have $\alpha_{\ot} (\dot{J}_t (p)) \neq \alpha_{\ot} (\dot{J}_{t'} (p))$. In fact if $t>t'$, then $\alpha_{\ot} (\dot{J}_t (p)) > \alpha_{\ot} (\dot{J}_{t'} (p))$. This is essentially because the Legendrian rays $R_{\nu_t}$ rotate faster than $R_{\mu_{t'}}$ by construction for the Legendrian isotopy from $J_{\text{thick}}$ to $J_{\text{thin}}$. For the extended isotopy from $J_{\text{thin}}$ to a point, this amounts to asserting that in the front projection in \autoref{fig:up_isotopy}(a), the $z$-coordinate of the lower branch\footnote{The front projection of the standard Legendrian unknot can be decomposed into the lower branch and upper branch by removing the two cusp points.} increases faster than that of the upper branch. These observations will be crucial in our construction of the loose unknot in a neighborhood of a general plastikstufe.

\section{plastikstufe implies overtwistedness in dimension 5} \label{sec:PS_to_OT}
In this section we prove that the germ of contact structure near a plastikstufe in dimension 5 is overtwisted. Together with the $h$-principle from \cite{BEM15}, this proves \autoref{thm:PS_criterion_new} in dimension 5.

Recall from \autoref{subsec:PS} that we have the standard overtwisted disk $D_{\ot} \subset (\RR^3_{\ot}, \xi_{\ot})$ equipped with the (overtwisted) contact form $\alpha_{\ot} = \cos rdz+r\sin rd\theta$. Consider the cotangent bundle $T^\ast S^1$ equipped with the standard Liouville form $pdq$ where $q \in \RR/\ZZ \cong S^1$ is the coordinate on the base and $p$ is the coordinate on the fiber. Then we form the 5-dimensional contact manifold $\RR^3_{\ot} \times T^\ast S^1$ with contact form $\alpha = \alpha_{\ot}-pdq$. The plastikstufe in dimension 5 is modeled on $D_{\ot} \times S^1 \subset \RR^3_{\ot} \times T^\ast S^1$ as a coisotropic submanifold. Write $P_{S^1} = D_{\ot} \times S^1$ and let $N_{\epsilon} (P_{S^1}) = \{ |z|<\epsilon, |p|<\epsilon \}$ be a small neighborhood of $P_{S^1}$. Then it suffices to show $(N_{\epsilon} (P_{S^1}), \alpha)$ is overtwisted.


It is convenient to think of $\RR^3_{\ot} \times T^\ast S^1$ as a trivial fibration over $S^1$ with fiber $\RR^3_{\ot}$ times the cotangent direction. We begin with some preparations concerning the 3-dimensional overtwisted fiber $\RR^3_{\ot}$. Let $K = \p D_{\ot}$ be the Legendrian unknot with $\tb(K)=0$. Pick a point $x \in K$ and let $K_0$ be a ``small'' (positive) stabilization of $K$ supported in a neighborhood of $x$ as shown in \autoref{fig:stabilization}(a). Note that the stabilizing loop in $K_0$ produces a kink in the front projection as depicted in \autoref{fig:kink}. In particular, by ``small'' we mean the action of the unique Reeb chord in the stabilization is less than $a$, for some small $0< a \ll \epsilon$ to be determined later. Equivalently, it also means the area (measured against $d\alpha_{\ot}$) of the disk enclosed by the stabilizing loop (after the projection to $D_{\ot}$) is smaller than $a$.

\begin{figure}[ht]
	\begin{overpic}[scale=.5]{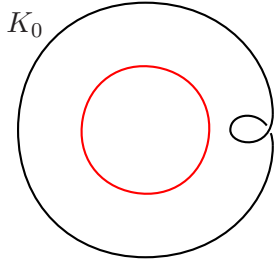}
		\put(-4,88){$K_0$}
	\end{overpic}
	\caption{A (positive) stabilization of $\p D_{\ot}$.}
	\label{fig:stabilization}
\end{figure}

Let $\gamma$ be the Legendrian ray emanating from the origin and passing through $x \in K$. Let $N(\gamma)$ be a standard tubular neighborhood of $\gamma$ in $D_{\ot} \times (-\epsilon,\epsilon)$ of radius greater than $a$. Observe that $\gamma$ intersects $K_0$ in a unique point $y$. Now pulling $y$ back along $\gamma$ to a point $y'$ sufficiently close to $0$ induces an ambient contact isotopy $\phi_t: D_{\ot} \times (-\epsilon,\epsilon) \to D_{\ot} \times (-\epsilon,\epsilon), t \in [0,1],$ which is supported in $N(\gamma)$. Here $\phi_0 = \text{id}$ and $\phi_1(y)=y'$. Then define $K_t = \phi_t(K_0), t \in [0,1]$, to be the Legendrian isotopy, supported in $N(\gamma)$. See \autoref{fig:unwind}. In particular, the projection of $K_1$ to $D_{\ot}$ is embedded and intersects the dividing set in two points. In addition, by further shrinking the stabilizing loop of $K_0$ if necessary, we can assume $\alpha_{\ot} (\dot{K}_t)$ is sufficiently small with respect to $\epsilon$. Here $\dot{K}_t$ means the time-derivative of the Legendrian isotopy.

\begin{figure}[ht]
	\begin{overpic}[scale=.5]{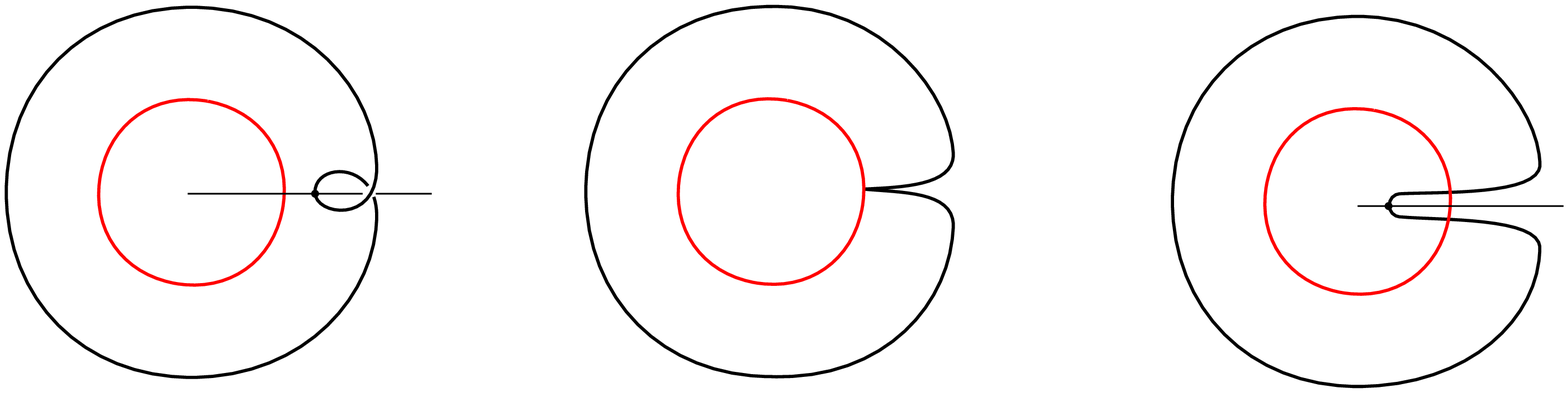}
		\put(26,14){$\gamma$}
		\put(18.4,13.5){\small{$y$}}
		\put(87.5,13){\small{$y'$}}
		\put(-1,22){$K_0$}
		\put(35,22){$K_{\frac{1}{2}}$}
		\put(74.5,22){$K_1$}
	\end{overpic}
	\caption{The projection to $D_{\ot}$ of a Legendrian isotopy.}
	\label{fig:unwind}
\end{figure}

According to the discussions in \autoref{subsec:type_I}, we may identify $K_1$ with $J_{\text{thick}}$, up to corner rounding (cf. Remark \autoref{rmk:corner_rounding}). Then there exists a type I Legendrian isotopy $K_t, t \in [1,2]$, supported in the given neighborhood of $D_{\ot}$ from $K_1$ to an arbitrarily ``thin'' Legendrian unknot $K_2=J_{\text{thin}}$ (cf. \autoref{fig:foliation1}).

With the above preparations, we are now ready to prove \autoref{thm:PS_criterion_new} in dimension 5. Fix a universal covering map $\RR \to \RR/\ZZ \cong S^1$. Slightly abusing notations, let $q$ be the coordinate on $\RR$ and $p$ be the dual coordinate on $T^\ast \RR$. First consider a neighborhood of $D_{\ot} \times [-1/10,1/10] \subset D_{\ot} \times S^1$ in $N_{\epsilon} (P_{S^1})$, inside of which we construct a Legendrian annulus $L_0 = K_0 \times [-1/10,1/10]$, which is Legendrian because $p=0$. Then by choosing $a$ sufficiently small, $L_0$ contains a loose chart.

Next, we consider the neighborhood of $D_{\ot} \times ([-1/5,-1/10] \cup [1/10,1/5])$ in $N_{\epsilon} (P_{S^1})$. Here we can construct the disjoint union of two annuli
	\begin{equation*}
		\widetilde{L}_1 := \bigcup_{\frac{1}{10} \leq |t| \leq \frac{1}{5}} (K_{10|t|-1} \times \{t\}).
	\end{equation*}
The Legendrian lift $L_1$ of $\widetilde{L}_1$, defined by $p = \alpha_{\ot} (\dot{K})$, can be assumed to be contained in $N_{\epsilon} (P_{S^1})$ by further shrinking $a$. To summarize, we fix $a$ small enough such that $L_0$ contains a loose chart and $L_1 \subset N_{\epsilon} (P_{S^1})$.

Observe that the projection of $L_0 \cup L_1$ to the base $S^1$ is an embedded interval $[-1/5,1/5] \subset S^1$. In order to close up $L_0 \cup L_1$ to get a Legendrian sphere, we use the type I Legendrian isotopy constructed in \autoref{subsec:type_I} as follows. Let $K_t, t \in [1,2]$, be the type I Legendrian isotopy supported in $D_{\ot} \times (-\epsilon,\epsilon)$ such that $K_1$ is identified with $J_{\text{thick}}$ and $K_2$ is identified with $J_{\text{thin}}$, up to corner rounding. Moreover, we can assume $J_{\text{thin}}$ is contained in a Darboux chart such that the front projection, as shown in \autoref{fig:thin}, contains a unique Reeb chord with action less than $\epsilon' \ll \epsilon$. Let us call $\epsilon'$ the action of $J_{\text{thin}}$.

\begin{figure}[ht]
	\begin{overpic}[scale=.4]{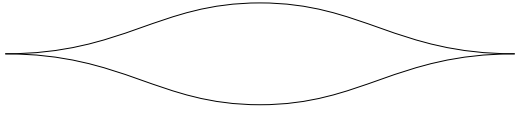}
		
	\end{overpic}
	\caption{The front projection of $J_{\text{thin}}$, after corner rounding.}
	\label{fig:thin}
\end{figure}

By compactness, there exists $N>0$ such that $\alpha_{\ot} (\dot{K}_t) < N$ for all $t \in [1,2]$. Choose $C > 0$ such that $N/(C-1) < \epsilon/2$. Consider the following annuli
	\begin{equation*}
		\widetilde{L}'_2 := \bigcup_{\frac{1}{5} \leq |t| \leq C} (K_{\frac{5t+5C-2}{5C-1}} \times \{t\}) \subset D_{\ot} \times (-\epsilon,\epsilon) \times \RR.
	\end{equation*}
The Legendrian lift of $\widetilde{L}'_2$, denoted by $L'_2$, is defined by $p=\alpha_{\ot} (\p_t K_{(5t+5C-2)/(5C-1)})$. Hence $L'_2$ is contained in $\epsilon$-neighborhood of $D_{\ot} \times \RR$. Abusing notations, we will also denote by $\widetilde{L}'_2$ and $L'_2$ their projections to $N_{\epsilon} (P_{S^1})$ modulo the $\ZZ$-action.

However, the Legendrian annuli $L'_2$ is not embedded in $N_{\epsilon} (P_{S^1})$ and the self-intersections occur precisely when $t \in \frac{1}{2}\ZZ \setminus \{0\}$. To remedy this, recall $\p_z$ is a contact vector field in $D_{\ot} \times (-\epsilon,\epsilon)$, and let $\phi_s$ be the time-$s$ flow of $\p_z$, which is defined near $D_{\ot} \times \{0\}$ for small $s$. Consider the following modification of $\widetilde{L}'_2$
\begin{equation} \label{eqn:lift_by_phi}
	\widetilde{L}_2 := \bigcup_{\frac{1}{5} \leq |t| \leq C} (\phi_{s(t)} (K_{\frac{5t+5C-2}{5C-1}}) \times \{t\}) \subset D_{\ot} \times (-\epsilon,\epsilon) \times \RR,
\end{equation}
where the function $s(t): [-C,-1/5] \cup [1/5,C] \to \RR$ is defined as follows. First for $t \in [1/5,C]$, we choose $s(t)$ such that $s(1/5)=0, s(C)=\epsilon/2$, $s$ is strictly increasing and $s'(1/5)=s'(C)=0$. Similarly for $t \in [-C,-1/5]$, choose $s(t)$ such that $s(-1/5)=0, s(-C)=\epsilon/3$, $s$ is strictly decreasing, $s'(-1/5)=s'(-C)=0$ and most importantly, $s(t) < s(-t)$ for all $t \in [-C,-1/5)$.

It is straightforward to check that the Legendrian lift of $\widetilde{L}_2$, denoted by $L_2$, still lies inside the $\epsilon$-neighborhood of $D_{\ot} \times \RR$. Moreover, the projection of $L_2$ to $N_{\epsilon} (P_{S^1})$ by modulo the $\ZZ$-action remains embedded by construction. Again, we abuse notation to call $L_2$ the resulting Legendrian annuli in $N_{\epsilon} (P_{S^1})$.

Finally to complete the construction of the loose Legendrian sphere, choose disjoint Darboux charts $U_{\text{right}} \supset \phi_{\epsilon/2} (K_2) \times [C]$ in $D_{\ot} \times (-\epsilon,\epsilon) \times [C]$ and $U_{\text{left}} \supset \phi_{\epsilon/3} (K_2) \times [-C]$ in $D_{\ot} \times (-\epsilon,\epsilon) \times [-C]$, respectively, where $[\pm C] \in S^1 \cong \RR/\ZZ$. This is possible as long as the difference in the $z$-coordinate of points in $K_2$ is less than $\epsilon/12$. In the following we only show how to cap-off the end at $\phi_{\epsilon/2} (K_2) \times [C]$ since the other end can be capped-off in the same way. Recall that in the above chosen Darboux chart $\phi_{\epsilon/2}(K_2)$ may be identified with $J_{\text{thin}}$ with action less than $\epsilon' \ll \epsilon$. Given $\epsilon'$ is small enough, one can cap-off $\phi_{\epsilon/2} (K_2) \times [C]$ by a standard Legendrian disk $Z_C$ in $U_{\text{right}} \times [C,C+1/5] \times (-\epsilon,\epsilon)$ whose front projection is shown in \autoref{fig:cap}. Here $[C,C+1/5]$ is regarded as an embedded interval in $S^1$ and $(-\epsilon,\epsilon)$ is a neighborhood of zero in the fiber of $T^\ast S^1$. In particular, the Legendrian disk $Z_C$ is constructed such that the projection to $D_{\ot}$ of its slice at every $q \in [C,C+1/5]$, if non-empty, is disjoint from any $K_t, t \in [0,2]$. Hence $Z_C$ is disjoint from $L_0 \cup L_1 \cup L_2$. Similarly one can construct the other capping disk $Z_{-C}$.

\begin{figure}[ht]
	\begin{overpic}[scale=.4]{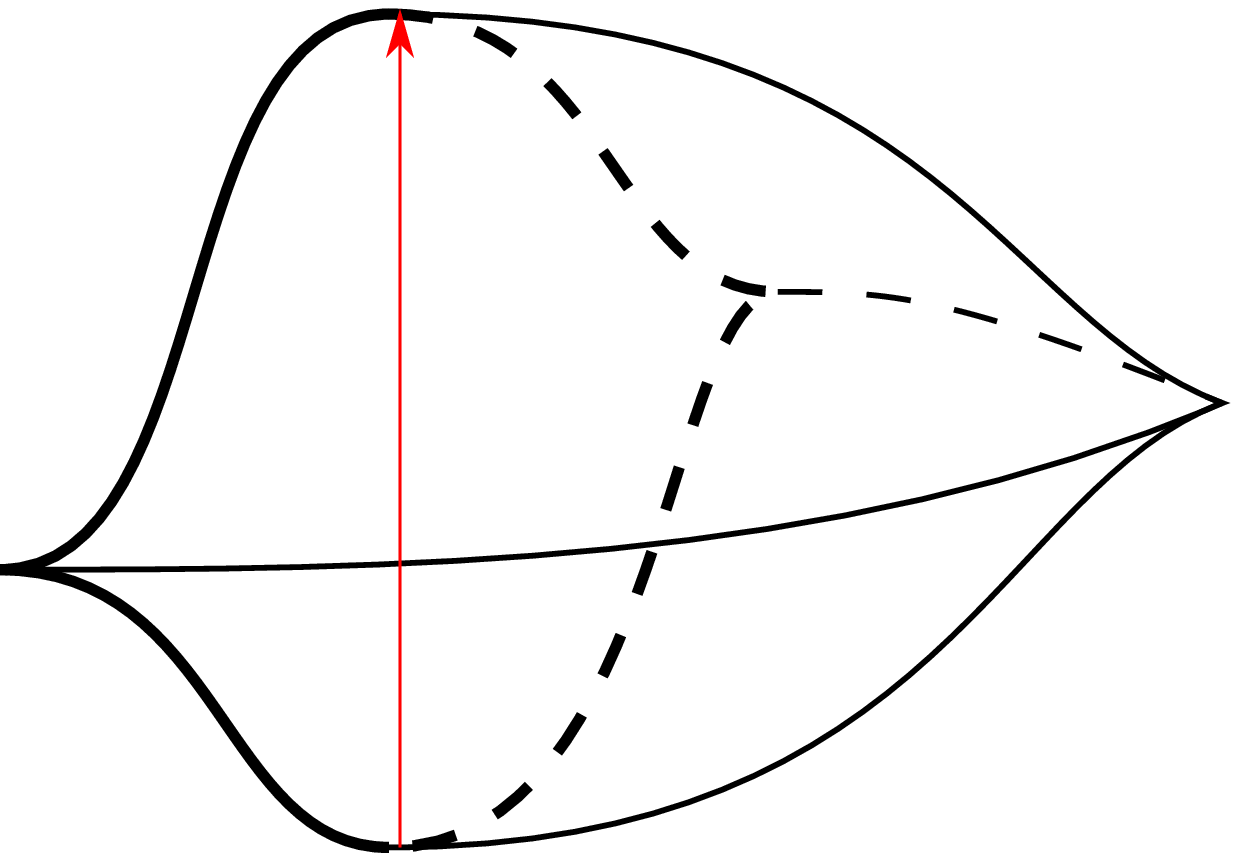}
		\put(-33,50){\small{$\phi_{\epsilon/2} (K_2) \times [C]$}}
		\put(72,63){\small{$Z_C$}}
		\put(36,35){\small{$<\epsilon'$}}
	\end{overpic}
	\caption{The front projection of the Legendrian capping disk $Z_{C}$}
	\label{fig:cap}
\end{figure}

Assembling all the pieces together, we have constructed a Legendrian sphere
	\begin{equation*}
		\Lambda := L_0 \cup L_1 \cup L_2 \cup Z_C \cup Z_{-C} \subset N_{\epsilon} (P_{S^1}),
	\end{equation*}
which is embedded by the above arguments.

To see that $\Lambda$ is loose, note that a loose chart exists in the $\epsilon$-neighborhood of $L_0$. Also observe that by construction the rest of the Legendrian sphere $L_1 \cup L_2 \cup Z_C \cup Z_{-C}$ does not intersect with the loose chart. To see that $L$ is also Legendrian isotopic to the standard Legendrian unknot, we first observe that $L_0 \cup L_1$ is induced by a Legendrian isotopy $K_{-10t-1}, t \in [-1/5,-1/10]$, followed by the constant isotopy $K_0$ for $t \in [-1/10,1/10]$, and then followed by the Legendrian isotopy $K_{10t-1}, t \in [1/10,1/5]$, which is the inverse to the first isotopy. It is therefore clear that $L_0 \cup L_1$ is Legendrian isotopic to the Legendrian cylinder $\text{Cyl}_{K_1} = K_1 \times [-1/5,1/5]$ induced by the constant isotopy $K_1$ for $t \in [-1/5,1/5]$, relative to the boundary. Hence we see that $\Lambda$ is Legendrian isotopic to
	\begin{equation*}
		\Lambda_1 := \text{Cyl}_{K_1} \cup L_2 \cup Z_C \cup Z_{-C} \subset N_{\epsilon} (P_{S^1}).
	\end{equation*}

Now for each $\tau \in [1,2]$, we can construct a Legendrian sphere $\Lambda_{\tau}$ by capping off the Legendrian cylinder $\text{Cyl}_{K_{\tau}} = K_{\tau} \times [-1/5,1/5]$ in the same way that $L_0 \cup L_1$ is capped off. More precisely, we define an isotopy of Legendrian spheres
	\begin{equation*}
		\Lambda_{\tau} := \text{Cyl}_{K_{\tau}} \cup L^{\tau}_2 \cup Z_{C_{\tau}} \cup Z_{-C_{\tau}},
	\end{equation*}
where $L^{\tau}_2$ is the Legendrian lift of
	\begin{equation*}
		\widetilde{L}^{\tau}_2 := \bigcup_{\frac{1}{5} \leq |t| \leq C_{\tau}} (\phi_{s_{\tau} (t)} (K_{\frac{(10-5\tau)t+5\tau C_{\tau} -2}{5C_{\tau} -1}}) \times \{t\}) \subset D_{\ot} \times (-\epsilon,\epsilon) \times \RR/\ZZ,
	\end{equation*}
and $Z_{C_{\pm \tau}}$ are the capping-off Legendrian disks analogous\footnote{Precisely, the difference between $Z_{C_{\pm \tau}}$ and $Z_{\pm C}$ is merely a translation in the $z$-direction.} to $Z_{\pm C}$.
Here $C_{\tau}$, as a function of $\tau$, is strictly decrasing such that $C_1=C, C_2=1/5$, and $s_{\tau} (t)$ is a function defined in the same way as $s(t)$ is defined right below \autoref{eqn:lift_by_phi} with $C$ replaced by $C_{\tau}$ and $\epsilon$ replaced by $\epsilon_{\tau}$ such that $\epsilon_{\tau}$ strictly decreases in $\tau$, $\epsilon_1=\epsilon$ and $\epsilon_2=0$. Therefore each $\Lambda_{\tau}, \tau \in [1,2]$, is an embedded Legendrian sphere, and $\Lambda$ is Legendrian isotopic to $\Lambda_2$.

It remains to observe that $\Lambda_2$ is Legendrian isotopic to the standard unknot. Indeed $\Lambda_2$ is nothing but the Legendrian cylinder $K_2 \times [-1/5,1/5]$ capped off by Legendrian disks $Z_{\pm 1/5}$ (cf. \autoref{fig:cap}). Hence $\Lambda_2$ is embedded in a Darboux chart in $D_{\ot} \times [-2/5,2/5]  \times (-\epsilon,\epsilon)^2 \subset N_{\epsilon} (P_{S^1})$. Then it is clear from the front projection that $\Lambda_2$ is the standard Legendrian unknot as desired. This completes the proof of \autoref{thm:PS_criterion_new} in dimension 5.

\section{(slit) blob implies overtwistedness in dimension 5} \label{sec:bLob_to_OT}
In this section we will generalize the argument in \autoref{sec:PS_to_OT} to give a proof of \autoref{thm:slit_bLob_criterion}, which immediately implies \autoref{thm:bLob_criterion}. The proof is divided into two steps.\\

\noindent
\textsc{Step 1.} \textit{Existence of a 2-dimensional overtwisted slice.}\\

The goal here is, roughly speaking, to find a 2-dimensional overtwisted slice in a slit bLob, cutting through the binding. More precisely, by a 2-dimensional overtwisted slice we mean a 2-disk $D$ transversely intersecting the binding in exactly one point, such that the characteristic foliation\footnote{The characteristic foliation on $D$ is the pointwise intersection between $TD$ and the contact hyperplanes.} on $D$ coincides with the characteristic foliation on a 2-dimensional overtwisted disk in contact 3-manifold (cf. \autoref{fig:OTdisk}). Such overtwisted slice exists in any plastikstufe simply by taking the product of the 2-dimensional overtwisted disk with a point in the binding. Note also that we do not need to assume the ambient contact manifold is 5-dimensional in this step.

Consider a slit bLob $\SB(\Sigma)$ associated with a compact manifold $\Sigma$. Write $\p \Sigma = S \sqcup S'$ such that $S$ is the (not necessarily connected) binding of the open book. In dimension 5, $S$ is a finite disjoint union of circles, so is $S'$. We first choose a local model for the contact structure in a tubular neighborhood $N(S)$ of $S$. Continuing using the notations from \autoref{defn:slit_bLob}, choose small $\delta>0$ and identify $N(S)$ with $S \times D_{\delta}$, where we think of
	\begin{equation*}
		D_{\delta} = \{ r \leq \delta \} \subset D_{\ot} \subset (\RR^3,\xi_{\ot}),
	\end{equation*}
as a small disk around the center of the overtwisted disk.

Pick a properly embedded arc $\gamma \subset \Sigma$ such that $\gamma$ transversely intersects $\p \Sigma$ and the two endpoints of $\gamma$ lie on $S$ and $S'$, respectively. Denote $\p_{\In} \gamma = \p \gamma \cap S$ and $\p_{\Out} \gamma = \p \gamma \cap S'$. 

Using notations from \autoref{subsec:PS}, let us write $\gamma^t = \gamma \times \{t\}$ for $t \in [0,1]$. By assumption $\gamma^0$ agrees with $\gamma^1$ when restricted to $\{ r \leq \delta \}$. Define the truncated arcs $\gamma^i_{\trun} = \gamma^i \cap \{ r \geq \delta \}, i=0,1$. Then we obtain a piecewise smooth isotropic loop
	\begin{equation*}
		\gamma^0_{\trun} \cup \gamma^1_{\trun} \cup (\p_{\Out} \gamma \times [0,1]),
	\end{equation*}
which bounds a ``cuspidal'' overtwisted disk $D^{\vee}_{\ot}$ in the sense that $\p D^{\vee}_{\ot}$ is Legendrian with a cusp singularity and the interior of $D^{\vee}_{\ot}$ is smooth with characteristic foliation coincides with the one on the 2-dimensional overtwisted disk. Here we implicitly round the corners at $\p_{\Out} \gamma \times \{0,1\}$ in the standard way (cf. Remark \autoref{rmk:corner_rounding}). Choose a 3-dimensional contact submanifold which contains $D^{\vee}_{\ot}$ in the interior. Then by either using Eliashberg's $h$-principle in \cite{Eli89} or examining the 3-dimensional contact germ near $D^{\vee}_{\ot}$, one can find a smooth overtwisted disk $D^{\circ}_{\ot}$ which is $C^0$-close to $D^{\vee}_{\ot}$. In fact, we may further assume without loss of generality that $D^{\circ}_{\ot}$ coincides with $D^{\vee}_{\ot}$ outside a small neighborhood of the cusp point $\gamma^0 \cap \{r=\delta\}$. See \autoref{fig:cusp}. We sometimes also write $D^{\circ}_{\ot}(\gamma)$ if we want to emphasize its dependence on $\gamma$. This is our substitute for the overtwisted disk slice in a plastikstufe.\\

\begin{figure}[ht]
	\begin{overpic}[scale=.4]{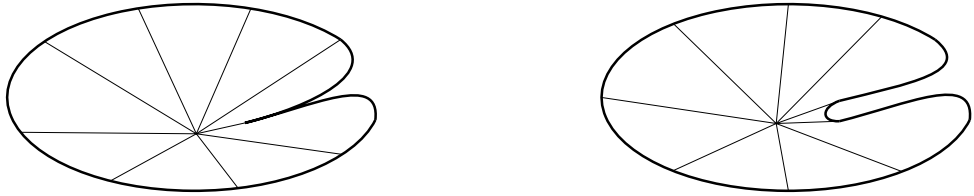}
		
	\end{overpic}
	\caption{The desingularization of the cuspidal overtwisted disk.}
	\label{fig:cusp}
\end{figure}

\noindent
\textsc{Notational remark.}
\textit{In order to make the following argument look parallel to the argument in \autoref{sec:PS_to_OT}, we reparametrize the $[0,1]$-direction of the open book by the angular variable $\theta \in [0,2\pi]$ such that page 0 is identified with $\{\theta=0\}$ and page 1 is identified with $\{\theta=2\pi\}$.\\}

Fix a point $u \in S$ and an arc $\gamma$ as above with $\p_{\In} \gamma=u$. Observe that the overtwisted disk $D^{\circ}_{\ot}$ is centered at $u$ in the sense that the characteristic foliation on $D^{\circ}_{\ot}$ has a unique elliptic singularity at $u$. Let $U \subset S$ be a small neighborhood of $u$. Then in the given contact neighborhood $N_{\epsilon} (\SB(\Sigma))$ of $\SB(\Sigma)$, there exists a contact embedding of a small neighborhood of $D^{\circ}_{\ot} \times U \subset \RR^3_{\ot} \times T^\ast U$, equipped with the product contact form, into $N_{\epsilon} (\SB(\Sigma))$. This is possible because $D^{\circ}_{\ot}$ is compact and thus the conformal factors are bounded. Just as in \autoref{sec:PS_to_OT}, we can construct a loose Legendrian annulus $S^1 \times U \subset N_{\epsilon} (\SB(\Sigma))$ such that for any $x \in \p U$, the Legendrian loop $S^1 \times x$ in the fiber $\RR^3_{\ot}$ can be identified with $J_{\text{thick}}$ (cf. \autoref{fig:stabilization}). More precisely, as discussed in \autoref{sec:Legendrian_isotopy}, it is uniquely determined by an arc $\gamma$ with $\p_{\In} \gamma=x$ and an angle between $0$ and $\pi$. We will denote this Legendrian loop by $\overline{\gamma}$ and the corresponding (small) angle by $\mangle \overline{\gamma} \in (0,\pi)$.

To summarize, we have so far found in any given neighborhood of a slit bLob $\SB(\Sigma)$ a 3-dimensional overtwisted slice cutting through the binding, and we use a neighborhood of the slice to construct a loose Legendrian annulus $S^1 \times U$ foliated by Legendrian loops in a neighborhood of $D_{\ot}$. Moreover, the boundary $S^1 \times \p U$ is the union of two copies of $J_{\text{thick}}$. So far the procedure is exactly the same as in the case of a plastikstufe.

Next we want to follow the strategy in \autoref{sec:PS_to_OT} to cap off the Legendrian annulus $S^1 \times U$ by wrapping $S^1 \times \p U$ around $S$ while performing the Legendrian isotopy from $J_{\text{thick}}$ to $J_{\text{thin}}$. This is, however, problematic because of the possibly non-trivial topology of the page $\Sigma$. The goal of the next step is to get around this problem when the page is 2-dimensional. At the moment it is not clear to the author whether this method can be generalized to higher dimensions or not.\\

\noindent
\textsc{Step 2.} \textit{Constructing loose Legendrian unknot assuming $\dim \Sigma=2$.}\\

From now on we assume $\dim \Sigma=2$, i.e., the page of the bLob is a compact surface. Pick a generic Morse function $f: \Sigma \to \RR$, satisfying the Morse-Smale condition, with only index 1 critical points such that $\p \Sigma = S \cup S'$ are regular level sets with $f(S)=0$ and $f(S')=1$. We need to take a closer look at the contact germ determined by $\SB(\Sigma)$. Namely, let $p_1, \cdots, p_m$ be the index 1 critical points and let $X_i \subset \Sigma$ be a collar neighborhood of the union the ascending and descending manifolds of $p_i$ for $i=1,\cdots,m$. It follows from the Morse-Smale condition\footnote{Namely, there exists no saddle-saddle connections. In our 2-dimensional case, the Morse-Smale condition is easy to achieve. But see, for example, \cite{Kup63,Sma63} for more general situations.} that $\p X_i \cap \p \Sigma = A_i \cup D_i$ for all $i$, where $A_i$ (resp. $D_i$) is a neighborhood in $S$ (resp. $S'$) of the intersection points between the ascending (resp. descending) manifold of $p_i$ with $S$ (resp. $S'$). We will call $\p X_i \cap \p \Sigma$ the horizontal boundary of $X_i$, and the complement the vertical boundary. Without loss of generality, we can assume that the vertical boundary of $X_i$ is tangent to the gradient trajectories. See \autoref{fig:saddle}. We will look at $\Sigma \setminus \bigcup_{i=1}^m X_i$ and $\bigcup_{i=1}^m X_i$ separately.

\begin{figure}[ht]
	\begin{overpic}[scale=.25]{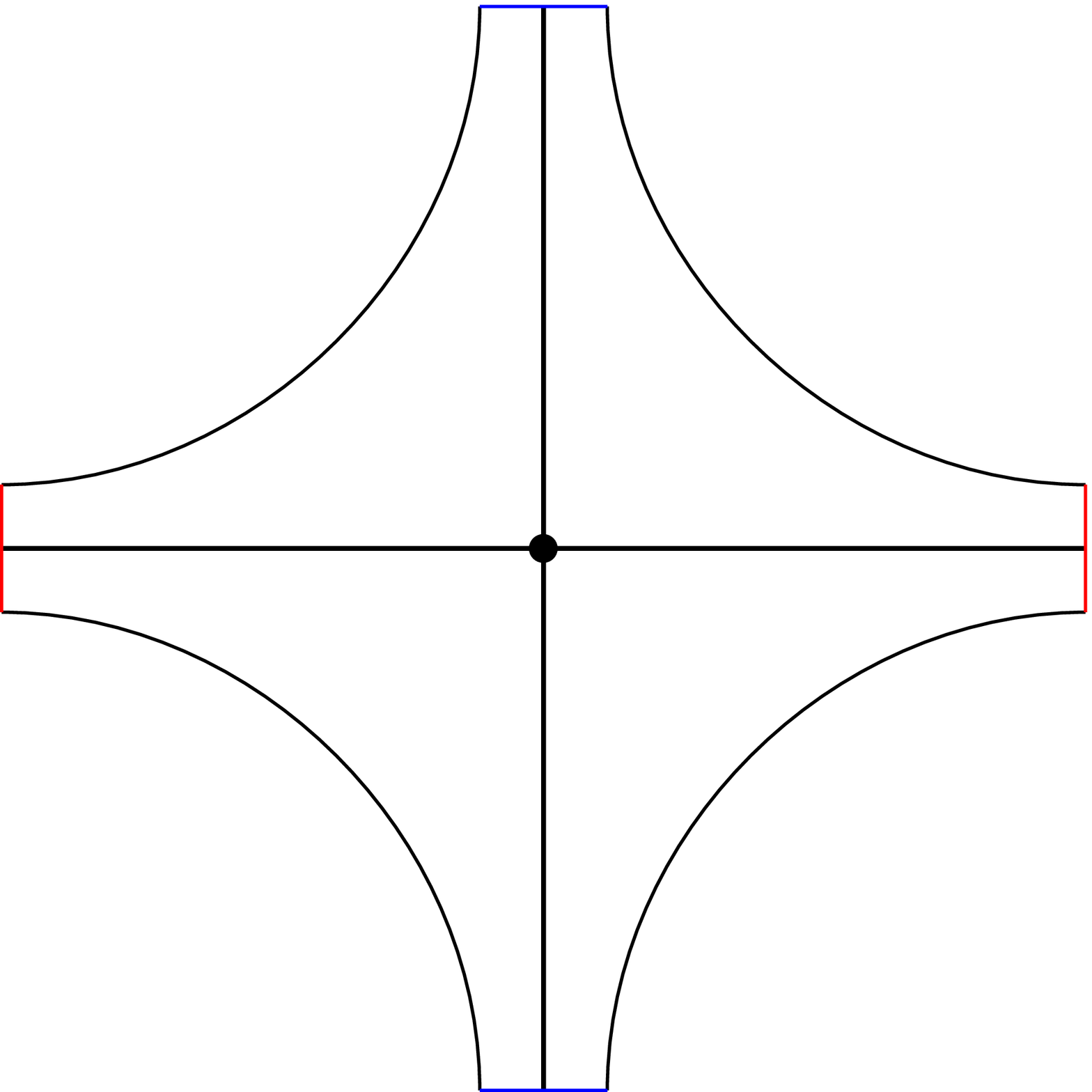}
		\put(52,45){$p_i$}
		\put(-10,48){$A_i$}
		\put(101,48){$A_i$}
		\put(46,-7){$D_i$}
		\put(46,102.5){$D_i$}
		\put(21,25){$\gamma_y$}
		\put(72,23){$\gamma_{z_1}$}
		\put(21,75){$\gamma_{z_3}$}
		\put(70,75){$\gamma_{z_2}$}
	\end{overpic}
	\caption{The standard embedding of $X_i$ into $\RR^2$.}
	\label{fig:saddle}
\end{figure}

The complement $\Sigma \setminus \bigcup_{i=1}^m X_i$ is clearly diffeomorphic to a finite disjoint union of rectangles, each of which is foliated by the gradient trajectories. Identify $\gamma$ with the gradient trajectories. Then each neighborhood $\gamma \times I \subset \Sigma$ of $\gamma$, where $I$ is an interval, determines a submanifold $\gamma \times I \times [0,2\pi]/\sim$ in $\SB(\Sigma)$, which has a neighborhood contactomorphic to a neighborhood of $D^{\circ}_{\ot}(\gamma) \times I \subset \RR^3_{\ot} \times T^\ast I$, equipped with the product contact form. This piece can be considered as a part of a plastikstufe by the inclusion $I \subset S^1$, where $S^1$ is the binding of the plastikstufe.

Now we turn to $\bigcup_{i=1}^m X_i$. Observe that the type I Legendrian isotopy in the fiber $\RR^3_{\ot}$ intersects neither the binding nor the Legendrian boundary of the open book. So we may slightly shrink $X_i$ near the horizontal boundaries such that $A_i$ and $D_i$ do not touch $\p \Sigma$. Abusing notations, we will still also it by $X_i$. For each $i$, it determines a submanifold $X_i \times [0,2\pi]/\sim$ in $\SB(\Sigma)$. In order to write down a contact form in coordinates, we embed $X_i \subset \RR^2_{q_1,q_2}$ such that the stable and unstable manifolds correspond to the $q_1$ and $q_2$-axis, respectively. See \autoref{fig:saddle}. Then the contact structure near $X_i \times [0,2\pi]/\sim$ can be identified with $T^\ast X_i \times [0,2\pi]_{\theta}$, equipped with the standard contact form $d\theta-p_1dq_1-p_2dq_2$.

Now we are in the position to construct the loose Legendrian unknot. Although we consider here only the low-dimensional case, we find it convenient to use a more general language, which may also serve as a warm-up for the next section in higher dimensions. Fix a complete Riemannian metric on $S$, which is a finite disjoint union of circles. Fix an arc $\gamma \subset \Sigma$ as above with $u = \p_{\In} \gamma \in S$ and a closed neighborhood $U \subset S$ of $u$. In the following we will denote by $\gamma_p$ the unique gradient trajectory from $S$ to $S'$ with $\p_{\In} \gamma_p = p \in S$. 

Choose one boundary point $x \in \p U$ (the other boundary point can be dealt with in the same way). Let $\mu_{u,x}$ be the oriented geodesic emanating from $u$ which passes through $x$. Then for some period of time at least, it induces a path of arcs $\gamma_{p_t}, 0 \leq t \leq \delta$, with $\p_{\In} \gamma_{p_0} = x$ and $\p_{\In} \gamma_{p_t} \in \mu_{u,x}$. Here, as before, each $\gamma_{p_t}$ coincides with an unbroken gradient trajectory. According to the above description of the contact structure, it also induces a path of overtwisted disks $D^{\circ}_{\ot} (\gamma_{p_t})$. So just as in the case of the plastikstufe, we can construct a Legendrian annulus by taking the Legendrian lift of the totality of the type I Legendrian isotopy from $\overline{\gamma}_x = \overline{\gamma}_{p_0}$ to $\overline{\gamma}_{p_{\delta}}$. Note also that the angle $\mangle \gamma_{p_t}$ increases as $t$ increases. By gluing this Legendrian annulus to the loose $S^1 \times U$ constructed in the end of Step 1, we obtain a ``longer'' Legendrian annulus which still contains a loose chart.

However, the (type I) Legendrian isotopy stops when the oriented geodesic $\mu_{u,x}$ hits $\p D_i \subset S$ for the first time for some $i$. Here we write $\p D_i \subset S$ with the original $D_i$ defined in the first paragraph in Step 2 in mind, instead of the one that shrinks into the interior of the page. We hope it will be clear from the context which version of $D_i$'s we use in the construction. Let $y \in \p D_i$ be the touching point and $\gamma_y$ be the corresponding gradient trajectory. By construction, $\gamma_y$ is contained in the vertical boundary of $X_i$. As discussed above, we shrink the horizontal boundary of $X_i$, as well as $\gamma_y$, slightly into the interior of the page and still call them $X_i$ and $\gamma_y$, respectively, by abusing notations. Let $X_i^\pm$ be the copies of $X_i$ on the pages corresponding to $\theta=\pm \mangle \gamma_y$, respectively. Then we construct a surface 
	\begin{equation} \label{eqn:bridge}
		\widetilde{T} = X_i^+ \cup X_i^- \cup (D_i \times \{ -\mangle \gamma_y \leq \theta \leq \mangle \gamma_y \}) \cup (A_i \times \{ \mangle \gamma_y \leq \theta \leq 2\pi-\mangle \gamma_y \}),
	\end{equation}
which has a canonical Legendrian lift $T$ in the standard contact neighborhood of $X_i \times [0,2\pi]/\sim$ described above. Topologically $T$ is nothing but a $S^2$ with four pairwise disjoint disks removed. Observe that $T$ can be glued to the previously constructed loose Legendrian annulus along $\overline{\gamma}_y$ to obtain a new Legendrian surface, say, $F$. Then $F$ has three new Legendrian boundary components $\overline{\gamma}_{z_i}, i=1,2,3$, which correspond to the three components $\gamma_{z_i}$ of the vertical boundary of $X_i$ other than $\gamma_y$. See \autoref{fig:saddle}. Once again the notation means $\p_{\In} \gamma_{z_i} = z_i \in S$ as usual. Now the geodesic flow on $S$ can be continued from $z_i, i=1,2,3$, separately. The above procedure can be repeated when the geodesic flow from some $z_i$ hits another (possibly the same) $X_j$ and so on. In this way we have constructed a Legendrian surface, still denoted by $F$, which is topologically a $S^2$ with finitely many pairwise disjoint disks removed. As in \autoref{sec:PS_to_OT}, the holes can be capped off in the standard way (cf. \autoref{fig:cap}) to obtain a Legendrian sphere $\overline{F} \cong S^2$ in the given neighborhood of $\SB(\Sigma)$ when all the boundary components $\gamma \subset \p F$ has sufficiently small angle $\mangle \gamma$ with respect to the given size of the neighborhood of $\SB(\Sigma)$. In other words, the capping-off operation can be done when all boundary components of $F$ can be identified with $J_{\text{thin}}$.

As in \autoref{sec:PS_to_OT}, the above constructed $\overline{F}$ may not be embedded in general. But this again causes no real problem as we are free to lift the type I Legendrian isotopy in the direction transverse to $D_{\ot}$ in the overtwisted fiber $\RR^3_{\ot}$.

By construction $\overline{F}$ is loose. We claim that $\overline{F}$ is also Legendrian isotopic to standard Legendrian unknot. The idea is essentially the same as in \autoref{sec:PS_to_OT}. Namely, we first undo the Legendrian isotopy from $K_0$ to $K_1$. Then we repeat the above construction with $K_1 \cong J_{\text{thick}}$, sitting in the fiber above $u \in S$, by the Legendrian loops in the type I isotopy from $K_1$ to $J_{\text{thin}}$. The only new ingredient here is to shrink the Legendrian sphere $\overline{F}$ back over $X_i$ for some $i$. This is most easily seen, when $\mangle{\gamma_y}$ is sufficiently small, in the front projection in a standard neighborhood of $X_i \times [0,2\pi]/\sim$ considered above. See \autoref{fig:bridge}. This finishes the proof of \autoref{thm:slit_bLob_criterion}, and hence \autoref{thm:bLob_criterion}.

\begin{figure}[ht]
	\begin{overpic}[scale=.3]{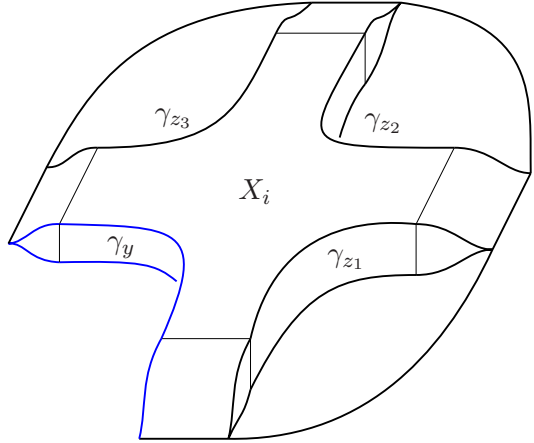}
		\put(19,36.5){$\gamma_y$}
		\put(61,34){$\gamma_{z_1}$}
		\put(68,60){$\gamma_{z_2}$}
		\put(28,61){$\gamma_{z_3}$}
		\put(44,46){$X_i$}
	\end{overpic}
	\caption{Half of the standard Legendrian unknot bounded by the union (in blue) of two copies of $\gamma_y$ and connecting arcs along the binding and the Legendrian boundary, respectively.}
	\label{fig:bridge}
\end{figure}

\section{plastikstufe implies overtwistedness in higher dimensions} \label{sec:PS_to_OT_general}

The ideas involved in proving \autoref{thm:PS_criterion_new} in higher dimensions are very similar to the 5-dimensional case explained in \autoref{sec:PS_to_OT}. There are, however, a few technical differences. First, we have to use the (extended) type II Legendrian isotopy introduced in \autoref{subsec:type_II} instead of the type I Legendrian isotopy in order to guarantee the (loose) Legendrian sphere constructed via ``wrapping around'' the binding is embedded. Second, extra care must be taken to make sure the ``wrapping'' procedure does not interfere with the loose chart. Here are the details.

In this case $S$ is a closed smooth manifold. Fix a complete Riemannian metric on $S$ and a point $u \in S$. We define the usual exponential map $\exp_u: T_u S \to S$ with respect to the metric. Let us denote by $B(r)$ the ball centered at 0 of radius $r$ in $T_u S$ and write $S(r) = \p B(r)$. Here the spheres $S(r)$ should not be confused with the binding $S$. By rescaling the metric if necessary, we may assume that the injective radius of $S$ is greater than 3. Then the restricted map $\exp_u|_{B(3)}: B(3) \to S$ is a smooth embedding. 

As in \autoref{sec:PS_to_OT}, we construct a loose Legendrian sphere, which turns out to be Legendrian isotopic to the standard unknot, in four stages.

Recall we have the trivial fiber bundle $\pi: N_{\epsilon} (P_S) \to S$ with fiber isomorphic to a neighborhood of $D_{\ot} \subset \RR^3_{\ot}$ times the cotangent fiber of $S$ of length less than $\epsilon$. In the following we will continue to use the notations from \autoref{sec:PS_to_OT}.

Firstly, upon $\exp_u (B(1)) \subset S$, we construct a Legendrian
	\begin{equation*}
		L_0 := K_0 \times \exp_u (B(1)).
	\end{equation*}
Here $\exp_u (B(1))$ is contained in the 0-section of $T^\ast S$. We choose the action $a$ of the stabilization in $K_0$ small enough such that $L_0$ contains a loose chart.

Secondly, upon $\exp_u (B(2) \setminus B(1)) \subset S$, we consider the submanifold
	\begin{equation*}
		\widetilde{L}_1 := \bigcup_{0 \leq t \leq 1} (K_t \times \exp_u (S(t+1))),
	\end{equation*}
which can be lifted to a Legendrian $L_1$ by setting $p = \alpha_{\ot} (\dot{K}_t)$, where $p$ is the coordinate on the fiber of $T^\ast S$. As in \autoref{sec:PS_to_OT}, we can choose $a$ sufficiently small such that $L_1 \subset N_\epsilon (P_S)$.

Thirdly, let $\phi_s$ be the time-$s$ flow of the contact vector field $\p_z$ in $D_{\ot} \times (-\epsilon,\epsilon)$. Then define
	\begin{equation*}
		\widetilde{L}_2 := \bigcup_{1 \leq t \leq 2} (\phi_{\frac{\epsilon}{2} (t-1)} (K_1) \times \exp_u (S(t+1))).
	\end{equation*}
This is well-defined since the maximal difference in the $z$-coordinates of points in $K_1$ is much less than $\epsilon$. Let $L_2$ be the Legendrian lift of $\widetilde{L}_2$ as before. Here $L_2$ is contained in $N_{\epsilon} (P_S)$ because $|\alpha_{\ot} (\p_t \phi_{\frac{\epsilon}{2}(t-1)} K_1)| < \epsilon$ for all $t \in [1,2]$.

Finally, we use the (extended) type II Legendrian isotopy from $K_1 \cong J_{\text{thick}}$ to $K_2 \cong J_{\text{thin}}$ and from $K_2 \cong J_{\text{thin}}$ to $K_3 = \{\text{pt}\}$ as constructed in \autoref{subsec:type_II}. Pick a large constant $C \gg 0$. Define
	\begin{equation*}
		\widetilde{L}_3 := \bigcup_{2 \leq t \leq C} (\phi_{\frac{\epsilon}{2}} (K_{\frac{2t+C-6}{C-2}}) \times \exp_u (S(t+1))),
	\end{equation*}
and let $L_3$ be the Legendrian lift of $\widetilde{L}_3$ in $N_{\epsilon} (P_S)$ which is well-defined as long as $C$ is sufficiently large. We claim that although $\widetilde{L}_3$ is typically not embedded, $L_3$ is a smooth embedded Legendrian submanifold in $N_\epsilon (P_S)$. Indeed, in the following, we rule out all the possibilities that $L_3$ may intersect itself, i.e., we show the self-intersections of $\widetilde{L}_3$ are removed by passing to the Legendrian lift.

\begin{enumerate}
	\item Let $v_1, v_2 \in T_u S$ be two vectors of length greater than 3 such that $\exp_u(v_1) = \exp_u(v_2) =:p$ and the Legendrian loops $K_{v_1}, K_{v_2} \subset D_{\ot} \times (-\epsilon,\epsilon)$ sitting above $p$ intersect non-trivially. For example, if $v_1, v_2$ have the same length, then $K_{v_1} = K_{v_2}$. In precise terms, we have $K_v := \phi_{\frac{\epsilon}{2}} (K_{\frac{2|v|+C-8}{C-2}})$. Now, since the corresponding geodesics passing through $\exp_u(v_1)$ and $\exp_u(v_2)$ are transverse to each other, the cotangent lift removes the self-intersection since $\alpha_{\ot} (\dot{K}) \neq 0$ everywhere on $K$.
	
	\item Suppose there is a closed geodesic $\gamma(t) = \exp_u (tv)$ emanating from $u$ in the direction of a unit vector $v$. Then by construction there exists (infinitely many) pairs $(t_0,t_1)$ such that the Legendrians $K_{t_0}, K_{t_1}$ intersect non-trivially in a point, say, $x$. Here $K_{t_i} := K_{t_i v}$. Then the fact $\alpha(\dot{K}_{t_0}(x)) \neq \alpha(\dot{K}_{t_1}(x))$ implies the cotangent lift removes the self-intersection. Hence after all, the Legendrian lift $L_3$ is embedded.
\end{enumerate}

Now we assemble the pieces together to define the Legendrian sphere $\Lambda = L_0 \cup L_1 \cup L_2 \cup L_3$ in $N_\epsilon (P_S)$. 

We claim that $\Lambda$ is loose. To see this, first recall that $L_0$ contains a loose chart. Then observe that $L_3$ does not interfere the loose chart because of the presence of $\phi_{\epsilon/2}$ in the definition of $L_3$ and the fact that the loose chart is contained in a $D_{\ot} \times (-2a,2a) \subset D_{\ot} \times (-\epsilon,\epsilon)$, where $a$ is the action of the stabilization and is much smaller than $\epsilon$. It remains to argue that $\Lambda$ is embedded. By the previous argument, it suffices to check that $L_3 \cap L_2 = \emptyset$. In fact, $\widetilde{L}_3$ intersects $\widetilde{L}_2$ non-trivially whenever, using the language from \autoref{subsec:type_II}, there is a slice $\phi_{\epsilon/2} (K_{t_0}) \times \{x\} \subset \widetilde{L}_3$ with $\nu_{t_0}=2\pi$ and $x \in \exp_u (B(3) \setminus B(2))$. The cotangent lifts, again, separate them since the derivative of the projection of $\phi_{\epsilon/2} (K_{t_0}) \times \{x\}$ to $D_{\ot}$, as $x$ moves in the geodesic direction, carries a nonzero $\p_{\theta}$-component, and $\alpha_{\ot} (\p_\theta) \neq 0$ on $0 < r < \pi$. This proves the claim.

The strategy of showing $\Lambda$ is Legendrian isotopic to the standard Legendrian unknot is identical to the one used in \autoref{sec:PS_to_OT}, and hence is omitted. This finishes the proof of \autoref{thm:PS_criterion_new}.

Lastly we give the proof of \autoref{cor:higher_dim} using the previous results.

\begin{proof}[Proof of \autoref{cor:higher_dim}]
In fact, we shall prove a stronger result by working with slit bLobs. Let $\SB(W \times F)$ be a slit bLob with page $W \times F$. In the following we will not distinguish between a slit bLob and the contact germ near the slit bLob.

Observe that $\SB(W \times F)$ may be identified with $\SB(F) \times T^\ast W$. It follows from \autoref{thm:bLob_criterion} that $\SB(F)$ is overtwisted. Hence by the $h$-principle from \cite{BEM15}, $\SB(F)$ contains a plastikstufe $P_{S^1}$. Therefore we see that $\SB(W \times F)$ also contains a plastikstufe $P_{W \times S^1}$, and hence is overtwisted by \autoref{thm:PS_criterion_new}.
\end{proof}

\section{Examples} \label{sec:example}

\subsection{Fibered connected sum}

The existence of plastikstufes in fibered connected sums along codimension 2 overtwisted submanifolds is well-known to the experts, see for example \cite{Pre07,CM16,Nie13}. Thanks to \autoref{thm:PS_criterion_new} we can now conclude that the resulting sum is overtwisted. Let us now detail how the construction of the plastikstufe works.

We begin by recalling the (contact) fibered connected sum operation following \cite[Section 7.4]{Gei08}. Let $(W,\xi)$ and $(W',\xi')$ be two (co-oriented) contact manifolds of dimension $\dim W' = \dim W-2$. Suppose there are disjoint contact embeddings $j_0,j_1: (W',\xi') \to (W,\xi)$ with trivial normal bundles, respectively.

Pick a trivialization $j_k(W') \times D^2_{2\epsilon} \subset W$ of the normal bundle of $j_k(W')$ for $k=0,1$, where $D^2_{2\epsilon}$ is the open disk of radius $2\epsilon$. By the standard contact neighborhood theorem, one can choose a contact form $\xi'=\ker\beta$ on $W'$ such that the following contact form
	\begin{equation*}
		\alpha_k = \beta+r^2 d\theta
	\end{equation*}
defines the restriction of $\xi$ to $j_k(W') \times D^2_{2\epsilon}$ for $k=0,1$, where $(r,\theta)$ is the standard polar coordinates on $D^2_{2\epsilon}$. 

Choose a function $f(r): (\epsilon/2,2\epsilon) \to \RR$ satisfying the following properties: 
	\begin{itemize}
		\item $f'(r)>0$,
		\item $f(r)=r^2$ on the interval $(\epsilon,2\epsilon)$,
		\item $f(r)=r^2-\epsilon^2/2$ on the interval $(\epsilon/2,\sqrt{3}\epsilon/2)$.
	\end{itemize}
Let $A_{\epsilon_0,\epsilon_1} = \{ \epsilon_0 \leq r \leq \epsilon_1 \} \subset D^2_{2\epsilon}$ denote the embedded annulus for any $0 < \epsilon_0 < \epsilon_1 <2\epsilon$. Then we can define a new contact structure on $j_k(W') \times A_{\epsilon/2,2\epsilon}$ by
	\begin{equation*}
		\widetilde{\alpha}_k = \beta + f(r)d\theta
	\end{equation*}
which agrees with $\alpha_k$ on $j_k(W') \times A_{\epsilon,2\epsilon}$.

Consider the diffeomorphism $\phi: j_0(W') \times A_{\epsilon/2,\sqrt{3}\epsilon/2} \to j_1(W') \times A_{\epsilon/2,\sqrt{3}\epsilon/2}$ defined by
	\begin{equation*}
		\phi(x,r,\theta) = (j_1 \circ j_0^{-1} (x), \sqrt{\epsilon^2-r^2}, -\theta).
	\end{equation*}
It is straightforward to check that $\phi^\ast (\widetilde{\alpha}_1) = \widetilde{\alpha}_0$. So we can construct the following manifold, known as the fibered connected sum,
	\begin{equation*}
		\#_\phi W = \big( W \setminus (j_0(W') \times D^2_{\epsilon/2}) \cup (j_1(W') \times D^2_{\epsilon/2}) \big)/\sim,
	\end{equation*}
where $(x,r,\theta) \sim \phi(x,r,\theta)$ for all $(x,r,\theta) \in j_0(W') \times A_{\epsilon/2,\sqrt{3}\epsilon/2}$. By the discussions above, the contact structure on $W$ naturally induces a contact structure on $\#_\phi W$.

We are now ready to state our first application of \autoref{thm:PS_criterion_new}.

\begin{theorem} \label{thm:OT_fibered_sum}
	Using the notations from above, if $(W',\xi')$ is overtwisted, then $\#_{\phi} W$ is also overtwisted. In other words, contact fibered connected sum along overtwisted contact submanifolds is overtwisted.
\end{theorem}

\begin{proof}
	Let $S^1(\epsilon/\sqrt{2}) \subset D^2_{2\epsilon}$ be the circle defined by $\{r=\epsilon/\sqrt{2}\}$. Then there is a contact embedding of a neighborhood of $W' \times S^1(\epsilon/\sqrt{2}) \subset W' \times T^\ast S^1(\epsilon/\sqrt{2})$, viewed as the 0-section, in $\#_{\phi} W$. Since $W'$ is overtwisted by assumption, it follows from the $h$-principle in \cite{BEM15} that there exists an embedded plastikstufe inside $W' \times S^1(\epsilon/\sqrt{2})$. Hence $\#_{\phi} W$ is overtwisted by \autoref{thm:PS_criterion_new}.
\end{proof}

\begin{remark}
There are several known constructions of codimension 2 overtwisted contact embeddings with trivial normal bundle into a tight contact manifold. For example, any (overtwisted) contact 3-manifold $(M,\xi)$ admits a contact embedding into its unit cotangent bundle as a graph whose normal bundle is trivial if $c_1(\xi)=0$. In fact the condition $c_1(\xi)=0$ may be removed by the ambient surgery techniques as explained in \cite{CPS16}. Another source of examples come from Bourgeois' construction \cite{Bou02} of contact structures on $M \times T^2$, where $M$ is equipped with a contact form supported by an open book decomposition $(\Sigma,\phi)$. For the sake of simplicity, assume $\dim M=3$. Then a simple calculation shows that if the binding $B = \p\Sigma \subset M$ is not contractible, i.e., each component of $B$ represents a nontrivial element in $\pi_1(M)$, then the Reeb vector field on $M \times T^2$ has no contractible orbit. This implies that $M \times T^2$ is tight in the light of \cite{AH09}.
\end{remark}

\subsection{Open book with monodromy a negative power of the Dehn twist}
In the light of the Giroux correspondence \cite{Gir02}, any contact structure is supported by an open book decomposition, say, $(S,h)$ where the page $S$ is a Weinstein domain and the monodromy $h$ is an exact symplectomorphism. It would be very interesting to see if the overtwistedness (or equivalently, tightness) of the contact  structure can be read off from the pair $(S,h)$. A complete solution to this problem in dimension 3 is given by the following remarkable theorem due to Honda, Kazez and Mati\'c.

\begin{theorem}[\cite{HKM07}] \label{thm:right_veering}
	A contact 3-manifold $(M,\xi)$ is tight if and only if all of its supporting open book $(S,h)$ has right-veering monodromy $h \in Aut(S,\p S)$.
\end{theorem}

In this case $S$ is a compact surface with boundary. While the precise definition of right-veering automorphisms of $S$ will not be needed here, we note that in the case $S = \DD^\ast S^1$, the disk cotangent bundle of $S^1$, any negative power of the Dehn twist $\tau^{-k}, k \geq 1$, is not right-veering. In particular it follows from \autoref{thm:right_veering} that the contact structure supported by $(\DD^\ast S^1,\tau^{-k})$ is overtwisted for all $k \geq 1$. The goal of this section is to generalize this fact to higher dimension.

\begin{theorem} \label{thm:dehn_twist}
	The contact structure supported by the open book $(\DD^\ast S^n, \tau^{-k})$ is overtwisted for all $k \geq 1$, where $\tau: \DD^\ast S^n \to \DD^\ast S^n$ is the (higher-dimensional) Dehn twist.
\end{theorem}

\begin{remark}
	The special case of \autoref{thm:dehn_twist} when $k=1$ is proved in \cite{CMP15}. It is also possible to give an alternative proof of \autoref{thm:dehn_twist} using the technique of higher-dimensional bypass attachment developed by Honda and the author.
\end{remark}

Abusing notations, let us also denote by $\tau: \DD^\ast S^n \to \DD^\ast S^n$ the (positive) Dehn twist in any dimension. See, for example, \cite{Sei08} and the references therein for more details. From now on, all spheres $S^n$ are assumed to be the standard unit sphere in the Euclidean space. Equipping $\DD^\ast S^n$ with the standard symplectic structure, it turns out that $\tau$ is an exact symplectomorphism. Moreover, for any great circle $S^1 \subset S^n$, the restriction $\tau|_{\DD^\ast S^1}$ is precisely the classical 2-dimensional Dehn twist. In particular, we have a contact embedding
	\begin{equation*}
		(\DD^\ast S^1, (\tau|_{\DD^\ast S^1})^{-k}) \subset (\DD^\ast S^n, \tau^{-k})
	\end{equation*}
of the 3-dimensional overtwisted $(\DD^\ast S^1, (\tau|_{\DD^\ast S^1})^{-k})$.

In the case $n=2m+1$ is odd, the proof of \autoref{thm:dehn_twist} is simple. Namely, consider the Hopf fibration
	\begin{equation*}
		S^1 \to S^{2m+1} \to \mathbb{CP}^m
	\end{equation*}
where all fibers are great circles in $S^{2m+1}$. By previous discussions, there exists an (3-dimensional) overtwisted disk $D_{\ot} \subset (\DD^\ast S^1, (\tau|_{\DD^\ast S^1})^{-k})$ for each great circle $S^1$. Hence $D_{\ot} \times \mathbb{CP}^m$ is an embedded plastikstufe in $(\DD^\ast S^n, \tau^{-k})$ with core $\mathbb{CP}^m$. Therefore \autoref{thm:dehn_twist} follows from \autoref{thm:PS_criterion_new}.

The idea of the proof of \autoref{thm:dehn_twist} for general $n$ is similar but one needs to be more careful about constructing the 3-dimensional overtwisted disk $D_{\ot}$ as the Hopf fibration no longer exists necessarily.

We start by gathering the minimal amount of knowledge from the theory of bypasses in 3-dimensional contact topology for our later purposes. Readers are referred to \cite{Hon00,Hon02,HKM05} for more details.

\subsubsection{Bypasses in contact 3-manifolds}
Given contact 3-manifold $(M,\xi)$, an embedded surface $\Sigma$ is called {\em convex}, in the sense of Giroux \cite{Gir91}, if there exists a vector field $v$ on $M$, whose flow preserves $\xi$, such that $v$ is everywhere transverse to $\Sigma$. For a convex surface $\Sigma$, we define the dividing set $\Gamma_{\Sigma} = \{ x \in \Sigma ~|~ v_x \subset \xi_x \}$. It turns out that $\Gamma_{\Sigma}$ is a properly embedded 1-submanifold of $\Sigma$ and the isotopy class of $\Gamma_{\Sigma}$ is independent of the choice of $v$. When there is no risk of confusion, we often write $\Gamma$ instead of $\Gamma_{\Sigma}$ for the dividing set.

A bypass $B$ is a convex half disk modeled on $\{ z \in \CC ~|~ |z| \leq 1, \text{Im}(z) \geq 0 \}$ with Legendrian boundary such that the dividing set $\Gamma = \{ z \in B ~|~ |z|=1/2 \}$ is a half circle intersecting $\p B$ in two points. It will be convenient to call $\p B \cap \{ \text{Im}(z)=0 \}$ the {\em straight boundary} of $B$ and $\p B \cap \{ \text{Im}(z)>0 \}$ the {\em curved boundary} of $B$. A bypass can be attached to a convex surface $\Sigma$ along the straight boundary if we choose a Legendrian arc $\gamma$ intersecting $\Gamma_{\Sigma}$ in three points: two being $\p \gamma$ and the third in the interior of $\gamma$. Such Legendrian arcs $\gamma$ are called {\em admissible}. See \autoref{fig:bypass}.

\begin{figure}[ht]
	\begin{overpic}[scale=.3]{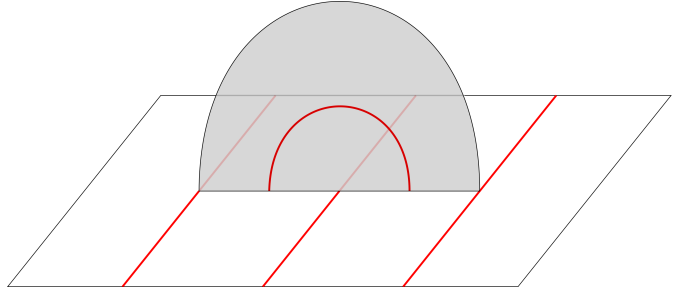}
		
	\end{overpic}
	\caption{A bypass attached to a convex surface.}
	\label{fig:bypass}
\end{figure}

Now given a bypass attached to a convex surface $\Sigma$, one can construct a contact structure on $\Sigma \times [0,1]$ such that both $\Sigma \times \{0\}$ and $\Sigma \times \{1\}$ are convex and $\Gamma_\Sigma = \Gamma_{\Sigma \times \{0\}}$. Roughly speaking, one can identify $\Sigma \times [0,1]$ with an appropriate neighborhood of $\Sigma \cup B$. It is also possible to describe $\Gamma_{\Sigma \times \{1\}}$ explicitly but since it is irrelevant for our purposes, we will leave it out.

Given an embedded convex surface $\Sigma \subset (M,\xi)$ and an admissible Legendrian arc $\gamma \subset \Sigma$, a bypass, if exists, can be attached to $\Sigma$ along $\gamma$ from either side of $\Sigma$. If $B$ is the bypass attached to $\Sigma$ along $\gamma$ from one side of $\Sigma$, then we denote by $\overline{B}$ the bypass attached from the other side along $\gamma$, and call it the {\em anti-bypass} of $B$.

The following facts will be crucial for us.
\begin{enumerate}
	\item Existence of trivial bypass \cite[Section 2.4]{Hon02}: Given a convex disk $D$ and an admissible Legendrian arc $\gamma$ as depicted in \autoref{fig:bypass_misc}(a), the bypass along $\gamma$ exists in an invariant neighborhood of $D$.

	\item Overtwisted disk: It follows almost from definition that a bypass and its anti-bypass can be glued along their straight boundaries to obtain an overtwisted disk. See \autoref{fig:bypass_misc}(b). As a corollary, if a convex surface $\Sigma$ divides $(M,\xi)$ into two connected components and a bypass along an admissible $\gamma \subset \Sigma$ exists in both components, then $(M,\xi)$ is overtwisted. This is, in fact, the motivating observation made in \cite{Hon02}.

	\item Bypass rotation \cite[Lemma 4.2]{HKM05}: Suppose we have a convex disk $D$ with dividing set equal to four arcs as depicted in \autoref{fig:bypass_misc}(c). If a bypass $B$ exists along the admissible $\gamma$, then there exists also a bypass $B'$ along $\gamma'$ obtained by ``left-rotating'' $\gamma$. See \autoref{fig:bypass_misc}(c). Moreover, the bypass $B'$ exists in a neighborhood of $D \cup B$.
\end{enumerate}

\begin{figure}[ht]
	\begin{overpic}[scale=.3]{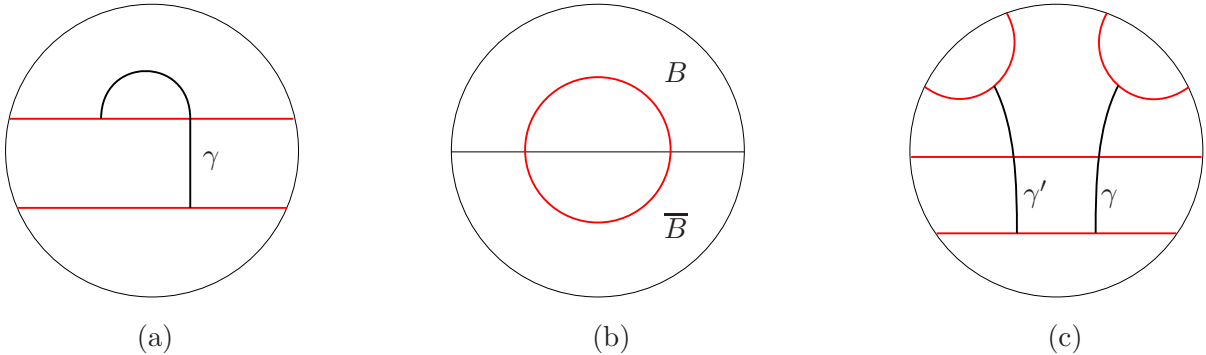}
		\put(11,-4){(a)}
		\put(49,-4){(b)}
		\put(87,-4){(c)}
		\put(16.5,11.2){$\gamma$}
		\put(55,18){$B$}
		\put(55,5){$\overline{B}$}
		\put(91.5,8){$\gamma$}
		\put(85,8){$\gamma'$}
	\end{overpic}
	\vspace{6mm}
	\caption{(a) The trivial bypass; (b) An overtwisted disk as the union of two bypasses; (c) Bypass rotation.}
	\label{fig:bypass_misc}
\end{figure}

\subsubsection{Proof of \autoref{thm:dehn_twist}}
With the above preparations, we are now ready to prove \autoref{thm:dehn_twist} in full generality. Recall we assume $S^n$ to be the unit sphere in $\RR^{n+1}$ with the north pole $\bdn$ and the south pole $\bds$. Let $S^{n-1} \subset S^n$ be the equator. Then for each $p \in S^{n-1}$, there is a unique great circle $S^1_p$ passing through $\bdn$ and $p$. In this way we cover $S^n$ by a $S^{n-1}$-family of great circles such that $S^1_p \cap S^1_q = \{ \bdn,\bds \}$ for any $p \neq \pm q \in S^{n-1}$. The same idea of using the family of great circles parametrized by the equator to constructed plastikstufe is employed in \cite[Theorem 5.3]{CMP15} to show that contact $(+1)$-surgery along loose Legendrian knots is overtwisted.

Fix a $p \in S^{n-1}$. Consider the 3-dimensional overtwisted contact submanifold $(\DD^\ast S^1_p, \tau^{-k}) \subset (\DD^\ast S^n, \tau^{-k})$. Here $\tau$ denotes the Dehn twist in appropriate dimensions. To avoid tedious notations, we will keep in mind the point $p$ but drop $p$ from our notations from now on. The next goal is to exhibit an explicit overtwisted disk in $(\DD^\ast S^1, \tau^{-k})$ using the technique of bypasses discussed previously. This strategy is borrowed from \cite{HKM07}.

First let us write $(\DD^\ast S^1, \tau^{-k}) = (\DD^\ast S^1 \times [0,1/2]) \cup_\phi (\DD^\ast S^1 \times [1/2,1])$ as the union of two solid tori, where the gluing map $\phi: \p (\DD^\ast S^1 \times [0,1/2]) \to \p (\DD^\ast S^1 \times [1/2,1])$ is defined by
	\begin{align*}
		\phi &= \tau^k: \DD^\ast S^1 \times \{0\} \to \DD^\ast S^1 \times \{1\}, \quad\text{and} \\
		\phi &= \text{id}: \DD^\ast S^1 \times \{1/2\} \to \DD^\ast S^1 \times \{1/2\}.
	\end{align*}
Here, as in the definition of open books, we also identify $\p (\DD^\ast S^1) \times \{t\} \sim \p (\DD^\ast S^1) \times \{t'\}$ for any $t,t' \in [0,1/2]$ or $t,t' \in [1/2,1]$, and simply call it $\p (\DD^\ast S^1)$. Then it is clear that both solid tori have convex boundary with dividing set $\Gamma = \p (\DD^\ast S^1)$.

Consider the solid torus $\DD^\ast S^1 \times [0,1/2]$. Let $\gamma \subset \p (\DD^\ast S^1 \times [0,1/2])$ be an admissible Legendrian arc contained in a neighborhood of $\DD^\ast_p S^1 \times [0,1/2]$ as depicted in \autoref{fig:two_bypass}(a). Then a bypass $B$ along $\gamma$ exists in a neighborhood of $\DD^\ast_p S^1 \times [0,1/2] \subset \DD^\ast S^1 \times [0,1/2]$. This follows essentially from the existence of the trivial bypass discussed above (cf. \cite[Fig.7]{HKM07}). It is convenient to write $\gamma = \gamma_+ \cup \gamma_-$ where $\gamma_+ = \gamma \cap (\DD^\ast S^1 \times \{1/2\})$ and $\gamma_- = \gamma \cap (\DD^\ast S^1 \times \{0\})$.

\begin{figure}[ht]
	\begin{overpic}[scale=.3]{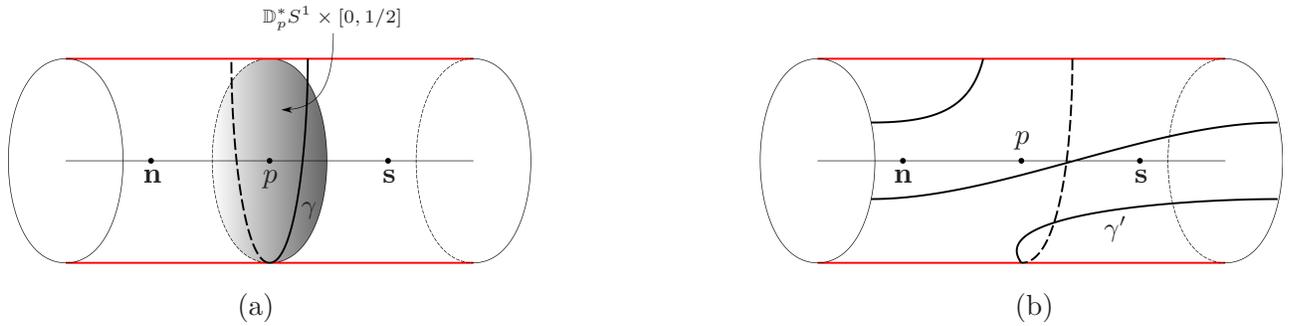}
		\put(10.6,6.4){$\bdn$}
		\put(20,6.4){$p$}
		\put(29.4,6.4){$\bds$}
		\put(20,19){\tiny{$\DD^\ast_p S^1 \times [0,1/2]$}}
		\put(23,4){\small{$\gamma$}}
		\put(69.5,6.4){$\bdn$}
		\put(88.4,6.4){$\bds$}
		\put(79,9.5){$p$}
		\put(86,2.2){\small{$\gamma'$}}
		\put(18,-4){(a)}
		\put(79,-4){(b)}
	\end{overpic}
	\vspace{5mm}
	\caption{(a) Existence of bypass along $\gamma$ in $\DD^\ast S^1 \times [0,1/2]$; (b) Existence of bypass along $\gamma'$ in $\DD^\ast S^1 \times [1/2,1]$.}
	\label{fig:two_bypass}
\end{figure}

Next we turn to the other solid torus $\DD^\ast S^1 \times [1/2,1]$. The corresponding admissible arc $\gamma' = \phi(\gamma)$ is as depicted in \autoref{fig:two_bypass}(b) (in the case $k=2$). Specifically, if we write $\gamma' = \gamma'_+ \cup \gamma'_-$ where $\gamma'_+ = \gamma' \cap (\DD^\ast S^1 \times \{1\})$ and $\gamma'_- = \gamma' \cap (\DD^\ast S^1 \times \{1/2\})$, then $\gamma'_+ = \tau^k (\gamma_-)$ and $\gamma'_- = \gamma_+$. By applying the bypass rotation $k$ times, one can show that a bypass $\overline{B}$ along $\gamma'$ exists in $\DD^\ast S^1 \times [1/2,1]$. 

As the notations suggest, the bypasses $B$ and $\overline{B}$ are anti-bypasses to each other. Hence $D_{\ot}(p) = B \cup \overline{B}$ is an overtwisted disk in $(\DD^\ast S^1_p, \tau^{-k})$. Moreover, it follows from the constructions that $D_{\ot}(p)$ is disjoint from the transverse knot/link\footnote{The connectivity depends on the parity of $k$.} $(\bdn \times [0,1]) \cup (\bds \times [0,1])$ and from $D_{\ot}(-p)$ in $(\DD^\ast S^1_p, \tau^{-k})$. Here we recall $S^1_p = S^1_{-p}$ by construction.

Finally, we take
	\begin{equation*}
		P_S = \bigcup_{p \in S^{n-1}} D_{\ot}(p),
	\end{equation*}
which is an embedded plastikstufe with core $S = S^{n-1}$. Hence \autoref{thm:dehn_twist} follows from \autoref{thm:PS_criterion_new}.

\bibliography{mybib}

\providecommand{\bysame}{\leavevmode\hbox to3em{\hrulefill}\thinspace}
\providecommand{\MR}{\relax\ifhmode\unskip\space\fi MR }
\providecommand{\MRhref}[2]{%
  \href{http://www.ams.org/mathscinet-getitem?mr=#1}{#2}
}
\providecommand{\href}[2]{#2}
\begin{thebibliography}{MNPS13}

\bibitem[Adaa]{Ada1610}
Jiro Adachi, \emph{{G}eneralizations of twists of contact structures to
  higher-dimensions via round surgery}, preprint, arXiv:1610.09672.

\bibitem[Adab]{Ada1609}
\bysame, \emph{{P}lastikstufe with toric core}, preprint, arXiv:1609.00837.

\bibitem[AH09]{AH09}
Peter Albers and Helmut Hofer, \emph{On the {W}einstein conjecture in higher
  dimensions}, Comment. Math. Helv. \textbf{84} (2009), no.~2, 429--436.
  \MR{2495800}

\bibitem[BEM15]{BEM15}
Matthew~Strom Borman, Yakov Eliashberg, and Emmy Murphy, \emph{Existence and
  classification of overtwisted contact structures in all dimensions}, Acta
  Math. \textbf{215} (2015), no.~2, 281--361. \MR{3455235}

\bibitem[Bou02]{Bou02}
Fr{\'e}d{\'e}ric Bourgeois, \emph{Odd dimensional tori are contact manifolds},
  Int. Math. Res. Not. (2002), no.~30, 1571--1574. \MR{1912277}

\bibitem[BvK10]{BvK10}
Fr\'ed\'eric Bourgeois and Otto van Koert, \emph{Contact homology of
  left-handed stabilizations and plumbing of open books}, Commun. Contemp.
  Math. \textbf{12} (2010), no.~2, 223--263. \MR{2646902}

\bibitem[Cas]{Cas}
Roger Casals, \emph{Private communication}.

\bibitem[CE12]{CE2012}
Kai Cieliebak and Yakov Eliashberg, \emph{From {S}tein to {W}einstein and
  back}, American Mathematical Society Colloquium Publications, vol.~59,
  American Mathematical Society, Providence, RI, 2012, Symplectic geometry of
  affine complex manifolds. \MR{3012475}

\bibitem[CM16]{CM16}
Roger Casals and Emmy Murphy, \emph{Contact topology from the loose viewpoint},
  Proceedings of the {G}\"okova {G}eometry-{T}opology {C}onference 2015,
  G\"okova Geometry/Topology Conference (GGT), G\"okova, 2016, pp.~81--115.
  \MR{3526839}

\bibitem[CMP]{CMP15}
Roger Casals, Emmy Murphy, and Francisco Presas, \emph{{G}eometric criteria for
  overtwistedness}, preprint, arXiv:1503.06221.

\bibitem[CPS16]{CPS16}
Roger Casals, Francisco Presas, and Sheila Sandon, \emph{Small positive loops
  on overtwisted manifolds}, J. Symplectic Geom. \textbf{14} (2016), no.~4,
  1013--1031. \MR{3601882}

\bibitem[Eli89]{Eli89}
Y.~Eliashberg, \emph{Classification of overtwisted contact structures on
  {$3$}-manifolds}, Invent. Math. \textbf{98} (1989), no.~3, 623--637.
  \MR{1022310}

\bibitem[EP11]{EP11}
John~B. Etnyre and Dishant~M. Pancholi, \emph{On generalizing {L}utz twists},
  J. Lond. Math. Soc. (2) \textbf{84} (2011), no.~3, 670--688. \MR{2855796}

\bibitem[Gei97]{Gei97}
Hansj{\"o}rg Geiges, \emph{Constructions of contact manifolds}, Math. Proc.
  Cambridge Philos. Soc. \textbf{121} (1997), no.~3, 455--464. \MR{1434654}

\bibitem[Gei08]{Gei08}
\bysame, \emph{An introduction to contact topology}, Cambridge Studies in
  Advanced Mathematics, vol. 109, Cambridge University Press, Cambridge, 2008.
  \MR{2397738}

\bibitem[Gir91]{Gir91}
Emmanuel Giroux, \emph{Convexit\'e en topologie de contact}, Comment. Math.
  Helv. \textbf{66} (1991), no.~4, 637--677. \MR{1129802}

\bibitem[Gir02]{Gir02}
\bysame, \emph{G\'eom\'etrie de contact: de la dimension trois vers les
  dimensions sup\'erieures}, Proceedings of the {I}nternational {C}ongress of
  {M}athematicians, {V}ol. {II} ({B}eijing, 2002), Higher Ed. Press, Beijing,
  2002, pp.~405--414. \MR{1957051}

\bibitem[Gro85]{Gro85}
M.~Gromov, \emph{Pseudoholomorphic curves in symplectic manifolds}, Invent.
  Math. \textbf{82} (1985), no.~2, 307--347. \MR{809718}

\bibitem[HKM05]{HKM05}
Ko~Honda, William~H. Kazez, and Gordana Mati{\'c}, \emph{Pinwheels and
  bypasses}, Algebr. Geom. Topol. \textbf{5} (2005), 769--784 (electronic).
  \MR{2153107}

\bibitem[HKM07]{HKM07}
\bysame, \emph{Right-veering diffeomorphisms of compact surfaces with
  boundary}, Invent. Math. \textbf{169} (2007), no.~2, 427--449. \MR{2318562}

\bibitem[Hon00]{Hon00}
Ko~Honda, \emph{On the classification of tight contact structures. {I}}, Geom.
  Topol. \textbf{4} (2000), 309--368. \MR{1786111}

\bibitem[Hon02]{Hon02}
\bysame, \emph{Gluing tight contact structures}, Duke Math. J. \textbf{115}
  (2002), no.~3, 435--478. \MR{1940409}

\bibitem[Hua15]{Hua15}
Yang Huang, \emph{On {L}egendrian foliations in contact manifolds {I}:
  {S}ingularities and neighborhood theorems}, Math. Res. Lett. \textbf{22}
  (2015), no.~5, 1373--1400.

\bibitem[Kup63]{Kup63}
Ivan Kupka, \emph{Contribution \`a la th\'eorie des champs g\'en\'eriques},
  Contributions to Differential Equations \textbf{2} (1963), 457--484.
  \MR{0165536}

\bibitem[MNPS13]{MNPS13}
Emmy Murphy, Klaus Niederkr{\"u}ger, Olga Plamenevskaya, and Andr{\'a}s~I.
  Stipsicz, \emph{Loose {L}egendrians and the plastikstufe}, Geom. Topol.
  \textbf{17} (2013), no.~3, 1791--1814. \MR{3073936}

\bibitem[MNW13]{MNW13}
Patrick Massot, Klaus Niederkr{\"u}ger, and Chris Wendl, \emph{Weak and strong
  fillability of higher dimensional contact manifolds}, Invent. Math.
  \textbf{192} (2013), no.~2, 287--373. \MR{3044125}

\bibitem[Mur]{Mur12}
Emmy Murphy, \emph{{L}oose {L}egendrian embeddings in high dimensional contact
  manifolds}, preprint, arXiv:1201.2245.

\bibitem[Nie06]{Nie06}
Klaus Niederkr{\"u}ger, \emph{The plastikstufe---a generalization of the
  overtwisted disk to higher dimensions}, Algebr. Geom. Topol. \textbf{6}
  (2006), 2473--2508. \MR{2286033}

\bibitem[Nie16]{Nie13}
\bysame, \emph{{O}n fillability of contact manifolds}, Habilitation {\`a}
  diriger des researches, Toulouse (2016).

\bibitem[NP10]{NP2010}
Klaus Niederkr{\"u}ger and Francisco Presas, \emph{Some remarks on the size of
  tubular neighborhoods in contact topology and fillability}, Geom. Topol.
  \textbf{14} (2010), no.~2, 719--754. \MR{2602849}

\bibitem[Pre07]{Pre07}
Francisco Presas, \emph{A class of non-fillable contact structures}, Geom.
  Topol. \textbf{11} (2007), 2203--2225. \MR{2372846}

\bibitem[Sei08]{Sei08}
Paul Seidel, \emph{Lectures on four-dimensional {D}ehn twists}, Symplectic
  4-manifolds and algebraic surfaces, Lecture Notes in Math., vol. 1938,
  Springer, Berlin, 2008, pp.~231--267. \MR{2441414}

\bibitem[Sma63]{Sma63}
S.~Smale, \emph{Stable manifolds for differential equations and
  diffeomorphisms}, Ann. Scuola Norm. Sup. Pisa (3) \textbf{17} (1963),
  97--116. \MR{0165537}

\end{thebibliography}
\bibliographystyle{amsalpha}

\end{document}